\documentclass[reqno]{amsart}

\usepackage{amsmath,amssymb,amscd, accents} \usepackage{graphicx}
\DeclareGraphicsExtensions{.eps} \usepackage{mathrsfs}
\usepackage[mathcal]{eucal} 
\usepackage{esint}

\newtheorem{Thm}{Theorem}{\bfseries}{\itshape}
\newtheorem*{Thm*}{Theorem}{\bfseries}{\itshape}
\newtheorem{Cor}{Corollary}{\bfseries}{\itshape}
\newtheorem{Prop}[Cor]{Proposition}{\bfseries}{\itshape}
\newtheorem*{Prop*}{Proposition}{\bfseries}{\itshape}
\newtheorem{Lem}[Cor]{Lemma}{\bfseries}{\itshape}
\newtheorem*{Lem*}{Lemma}{\bfseries}{\itshape}
\newtheorem{Fact}[Cor]{Fact}{\bfseries}{\itshape}
{\bfseries}{\itshape}
\newtheorem{Def}[Cor]{Definition}{\bfseries}{\rmfamily}
\newtheorem{Ex}[Cor]{Example}{\scshape}{\rmfamily}
\newtheorem{Rem}[Cor]{Remark}{\scshape}{\rmfamily}
{\bfseries}{\itshape}

\renewcommand\ge{\geqslant} \renewcommand\le{\leqslant}
\let\tildeaccent=\~ \let\hataccent=\^
\renewcommand\~[1]{\widetilde{#1}}

\def\<{\left<} \def\>{\right>} \def\({\left(} \def\){\right)}

\def\abs#1{\left\vert #1 \right\vert} \def\norm#1{\left\Vert #1
  \right\Vert} 

\let\parasymbol=\S \def\secref#1{\parasymbol\ref{#1}}
 \def\pd#1#2{\tfrac{\partial#1}{\partial#2}}

\let\polishL=l \def\Zoladek.{\.Zol\c adek}

\def\Re{\operatorname{Re}} \def\Im{\operatorname{Im}}
 \def\dist{\operatorname{dist}}
 
 \def\SL{\operatorname{SL}}
\def\Sp{\operatorname{Sp}}
\def\PSL{\operatorname{PSL}}
\def\etc.{\emph{etc}.}

\def\Sing{\operatorname{Sing}}

\def\:{\colon} \def\R{{\mathbb R}} \def\C{{\mathbb C}} \def\Z{{\mathbb
    Z}} \def\N{{\mathbb N}} \def\Q{{\mathbb Q}} 
\def\H{{\mathbb H}}

\let\PolishL=\L % remember polish L
\def\L{{\mathbb L}}

 \def\e{\varepsilon} \def\S{\varSigma}
   
 \def\diag{\operatorname{diag}}
\def\poly{{\operatorname{poly}}}

 \def\d{\,\mathrm d}

 \def\Lojas.{\PolishL ojasiewicz}
 \def\cH{{\mathcal H}}
\def\cP{{\mathcal P}} \def\cR{{\mathcal R}}
\def\scS{{\mathscr S}}

  \def\cR{{\mathcal R}}
  
 \def\cD{{\mathcal D}}

\def\cO{{\mathcal O}}

\def\cG{{\mathcal G}}
\def\cA{{\mathcal A}}

\def\Aut{\operatorname{Aut}}

\def\rest#1{{\vert_{#1}}}

\def\w{\omega}

\def\Qa{\Q^{\mathrm{alg}}}

\def\alg{\mathrm{alg}}
\def\trans{\mathrm{trans}}

\def\RE{\mathrm{RE}}

\def\det{\operatorname{det}}

\def\fB{{\mathfrak B}}

\def\vf{{\mathbf f}}

\def\vx{{\mathbf x}}
\def\vy{{\mathbf y}}
\def\vz{{\mathbf z}}
\def\vw{{\mathbf w}}
\def\vp{{\mathbf p}}

\def\vc{{\mathbf c}}

\def\vj{{\mathbf j}}

\def\vphi{{\boldsymbol\phi}}
\def\vrho{{\boldsymbol\rho}}
\def\vF{{\mathbf F}}

\def\adj{\operatorname{adj}}

\def\NS{\scS}
\def\^#1{^{(#1)}{}}

\def\Hdr{H_\text{dR}}

\begin{document}

% +Title
\title{Density of algebraic points on Noetherian varieties}

\author{Gal Binyamini} 
\address{Weizmann Institute of Science, Rehovot, Israel}
\email{gal.binyamini@weizmann.ac.il}

\begin{abstract}
  Let $\Omega\subset\R^n$ be a relatively compact domain. A finite
  collection of real-valued functions on $\Omega$ is called a
  \emph{Noetherian chain} if the partial derivatives of each function
  are expressible as polynomials in the functions. A \emph{Noetherian
    function} is a polynomial combination of elements of a Noetherian
  chain. We introduce \emph{Noetherian parameters} (degrees, size of
  the coefficients) which measure the complexity of a Noetherian
  chain. Our main result is an explicit form of the Pila-Wilkie
  theorem for sets defined using Noetherian equalities and
  inequalities: for any $\e>0$, the number of points of height $H$ in
  the transcendental part of the set is at most $C\cdot H^\e$ where
  $C$ can be \emph{explicitly} estimated from the Noetherian
  parameters and $\e$.

  We show that many functions of interest in arithmetic geometry fall
  within the Noetherian class, including elliptic and abelian
  functions, modular functions and universal covers of compact Riemann
  surfaces, Jacobi theta functions, periods of algebraic integrals,
  and the uniformizing map of the Siegel modular variety $\cA_g$. We thus
  effectivize the (geometric side of) Pila-Zannier strategy for
  unlikely intersections in those instances that involve only compact
  domains.
\end{abstract}
%%-Title
\date{\today}
\maketitle

\section{Introduction}

\subsection{The (real) Noetherian class}

Let $\Omega_\R\subset\R^n$ be a bounded domain, and denote by
$\vx:=(x_1,\ldots,x_n)$ a system of coordinates on $\R^n$. A
collection of analytic functions
$\vphi:=(\phi_1,\ldots,\phi_\ell):\bar\Omega_\R\to\R^\ell$ is called a
(complex) \emph{real Noetherian chain} if it satisfies an
overdetermined system of algebraic partial differential equations,
\begin{equation}\label{eq:noetherian-sys}
  \pd{\phi_i}{x_j} = P_{i,j}(\vx,\vphi), \qquad
  \begin{aligned}
    i=1,\ldots,\ell \\ j=1,\ldots,n
  \end{aligned}  
\end{equation}
where $P_{i,j}$ are polynomials. We call $\ell$ the \emph{order} and
$\alpha:=\max_{i,j} \deg P_{i,j}$ the \emph{degree} of the chain. If
$P\in\R[\vx,\vy]$ is a polynomial of degree $\beta$ then
$P(\vx,\vphi):\Omega_\R\to\R$ is called a \emph{real Noetherian
  function} of degree $\beta$.

We call the set of common zeros of a collection of real Noetherian
functions of degree at most $\beta$ a \emph{real Noetherian variety}
of degree $\beta$. We call a set defined by a finite sequence of
Noetherian equations or inequalities a \emph{basic semi-Noetherian
  set}, and a finite union of such sets a \emph{semi-Noetherian} set.
We define the \emph{complexity} $\beta$ of a semi-Noetherian set (more
precisely the formula defining it) to be the maximum of the degrees of
the Noetherian functions appearing in the definition, plus the total
number of relations. We use an analogous definition for the complexity
of a semialgebraic set.

We define the \emph{Noetherian size} of $\vphi$, denoted $\NS(\vphi)$,
to be
\begin{equation}
  \NS(\vphi) := \max_{x\in\bar\Omega_\R}\max_{\substack{i=1,\ldots,\ell\\ j=1,\ldots,n}} \{ |x_j|,|\phi_i(\vx)|,\norm{P_{i,j}}_\infty \}.
\end{equation}
Here and below $\norm{P}_\infty$ denotes the maximum norm on the
coefficients of $P$. For simplicity of the notation we always assume
$\NS(\vphi)\ge2$. In this paper we will be concerned with the problem
of producing explicit estimates for some quantities associated to
Noetherian varieties and semi-Noetherian sets. When we say that a
quantity can be \emph{explicitly estimated} in terms of the Noetherian
parameters, we mean that it admits an explicit upper bound in terms of
the parameters $n,\ell,\alpha,\NS(\vphi),\beta$.

\subsection{Main statement}

For a set $A\subset\R^n$ we define the \emph{algebraic part}
$A^\alg$ of $A$ to be the union of all connected semialgebraic subsets
of $A$ of positive dimension. We define the \emph{transcendental part}
$A^\trans$ of $A$ to be $A\setminus A^\alg$. Recall that the
\emph{height} of a (reduced) rational number $\tfrac a b\in\Q$ is
defined to be $\max(|a|,|b|)$. More generally, for $\alpha\in\Qa$ we
denote by $H(\alpha)$ its absolute multiplicative height as defined in
\cite{bombieri:heights}. For a vector $\boldsymbol\alpha$ of algebraic
numbers we denote by $H({\boldsymbol\alpha})$ the maximum among the
heights of the coordinates. For a set $A\subset\Omega_\R$ we denote the
set of $\Q$-points of $A$ by $A(\Q):=A\cap\Q^n$ and denote
\begin{equation}
  A(\Q,H) := \{\vx\in A(\Q):H(\vx)\le H\}.
\end{equation}

Throughout the paper we let $C_j(\vphi,d)$ denote the asymptotic class
\begin{align}
  C_j(\vphi,d) &:= (\NS(\vphi))^{\exp^{\circ4j}(O(d^2))} & C(\vphi,d):=C_1(\vphi,d)
\end{align}
where it is understood that each occurrence may represent a different
function from the class. In all instances of asymptotic notation in
this paper, it is understood that the implied constants \emph{can be
  explicitly and straightforwardly estimated in terms of the
  Noetherian parameters}, even though we do not always produce
explicit expressions for the constants. The following is a basic form
of our main theorem.

\begin{Thm}\label{thm:main}
  Let $X\subset\Omega_\R$ be a semi-Noetherian set of complexity
  $\beta$ and $\e>0$. There exists a constant
  \begin{equation}
    N=C_n(\vphi,\beta\e^{1-n})
  \end{equation}
  such that for any $H\in\N$ we have
  \begin{equation}
    \#X^\trans(\Q,H) \le N\cdot H^\e.
  \end{equation}
\end{Thm}

Theorem~\ref{thm:main} is a direct corollary of the following more
general statement. First, we consider \emph{algebraic} points of a
fixed degree $k\in\N$ instead of rational points. Toward this end we
introduce the notation
\begin{align}
  A(k) &:= \{\vx\in A:[\Q(x_1):\Q],\ldots,[\Q(x_n):\Q]\le k\}, \\
  A(k,H) &:= \{\vx\in A(k):H(\vx)\le H\}.  
\end{align}
Second, we obtain a more accurate description of the part of $X^\alg$
where algebraic points of a given height may lie. Toward this end we
introduce the following notation.
\begin{Def}
  Let $A,W$ be two subsets of a topological space. We denote by
  \begin{equation}
    A(W) := \{ w\in W: W_w\subset A\}
  \end{equation}
  the set of points of $W$ such that $A$ contains the germ of $W$
  around $w$, i.e. such that $w$ has a neighborhood $U_w$ with
  $U_w\cap W\subset A$.
\end{Def}
In particular, when $W\subset\R^n$ is a connected positive dimensional
semialgebraic set then we have $X(W)\subset X^\alg$. With these
notations, the general form of our main theorem is as follows.

\begin{Thm}\label{thm:main-k}
  Let $X\subset\Omega_\R$ be a semi-Noetherian set of complexity
  $\beta$ and $\e>0$. There exists constants
  \begin{equation}
    d,N = C_{n(k+1)}(\vphi,\beta\e^{1-n})
  \end{equation}  
  with the following property. For every $H\in\N$ there exist at most
  $N H^\e$ smooth connected semialgebraic sets $\{S_\alpha\}$ of
  complexity at most $d$ such that
  \begin{equation}
    X(k,H) \subset \bigcup_\alpha X(S_\alpha).
  \end{equation} 
\end{Thm}

We remark that in Theorem~\ref{thm:main-k} we allow the asymptotic
constants to depend on the degree $k$ as well.

\subsection{Motivation}

\subsubsection{The Pila-Wilkie theorem and the Pila-Zannier strategy
  for problems of unlikely intersections}
Following the fundamental work of Bombieri and Pila
\cite{bombieri-pila}, Pila and Wilkie proved in \cite{pila-wilkie}
that for any set definable in an o-minimal structure, the number of
rational points of height $H$ in the transcendental part grows
sub-polynomially with $H$ in the sense of Theorem~\ref{thm:main}. In
other words, without the added condition of effectivity,
Theorem~\ref{thm:main} is already known in vast generality. Similarly,
a non-effective result similar in spirit to Theorem~\ref{thm:main-k},
valid for arbitrary definable sets, has been established in
\cite{pila:algebraic-points}.

Beyond the intrinsic interest in the study of density of rational
points on transcendental sets, this direction of research has
attracted considerable attention following the discovery of a
surprising link to various problems of unlikely intersections in
arithmetic geometry. The first and prototypical example of this link
was produced in Pila-Zannier's \cite{pz:manin-mumford} proof of the
Manin-Mumford conjecture (first proved by Raynaud \cite{raynaud:mm}).
We briefly recall the statement and strategy of proof to motivate the
following discussion.

Let $A$ be an abelian variety and $V\subset A$ an algebraic
subvariety, both defined over a number field. Suppose that $V$ does
not contain a coset of an infinite abelian subvariety. Then (a
particular case of) the Manin-Mumford conjecture asserts that the
number of torsion points is finite. The strategy of
\cite{pz:manin-mumford} proceeds as follows. Let $\pi:\C^g\to A$
denote the universal cover of $A$ and $\Omega\subset\C^g$ denote the
standard fundamental domain. Identify $\C^g$ with $\R^{2g}$ in such a
way that $\Omega$ corresponds to the unit cube, and observe that under
this identification the torsion points of order $H$ in $V$ correspond
to rational points of height $H$ in $X:=\pi^{-1}(V)\cap\Omega$. One
now obtains two competing estimates for $\#X(\Q,H)$:
\begin{enumerate}
\item One checks that under the assumptions on $V$ one has
  $X^\trans=X$. Thus by the Pila-Wilkie theorem $\#X(Q,H)$ grows
  sub-polynomially with $H$.
\item By a result of Masser \cite{masser:height}, if $p\in A$ is
  torsion of order $H$ then the number of its Galois conjugates is at
  least $cH^\delta$ for some $c,\delta>0$. Since $V$ is defined over a
  number field, a constant fraction of these conjugates belong to $V$,
  and we conclude that $\#X(\Q,H)>c'H^\delta$ for some $c'>0$.
\end{enumerate}
The inconsistency of these two estimates implies that for $H$
sufficiently large, $V$ contains no torsion points of order $H$. In
particular the number of torsion points is finite.

The Manin-Mumford conjecture has a prototypical form: given an
arithmetic condition on $p\in A$ (being torsion) and a geometric
condition ($p\in V$), the number of solutions is finite unless for
some ``obvious'' reasons (e.g. $V$ contains an abelian subvariety).
Various other problems of a similar prototype have been solved by
using the same basic Pila-Zannier strategy. We list two prominent
examples:
\begin{description}
\item[The Andr\'e-Oort conjecture] We describe the special case
  considered in \cite{andre:finitude} for simplicity. Consider the
  product $Y(N_1)\times Y(N_2)$ of two modular curves, and its
  irreducible algebraic subvariety $V$, and suppose $V$ is not defined
  by a \emph{modular polynomial}. Then the number of points
  $(p_1,p_2)\in V$ where both $p_1,p_2$ are CM-points (i.e. correspond
  to elliptic curves with complex multiplication) is finite
  \cite{andre:finitude}. A more general case of this statement, the
  \emph{Andr\'e-Oort conjecture} for modular curves (involving the
  products of an arbitrary number of modular curves as well as abelian
  varieties and complex tori) was proved using the Pila-Zannier
  strategy in \cite{pila:andre-oort}. The uniformization maps of
  modular curves play a key role in this proof. Note that since the
  fundamental domains of modular curves are never compact, the
  definable sets appearing in this proof are not subanalytic and the
  full strength of the Pila-Wilkie theorem in the o-minimal setting is
  required to study their behavior near the cusps. This proof was
  later extended, with significant effort on the Galois-theoretic
  side, to various other contexts involving Shimura varieties. We
  refer the reader to \cite{et:ao} for a survey of various developments
  in this area, and to \cite{tsimerman:ao-Ag} for the more recent
  unconditional proof of the Andr\'e-Oort conjecture for $\mathcal{A}_g$.
\item[Torsion anomalous points] Consider the two points
  \begin{align}\label{eq:mz-PQ}
    P(\lambda) &= (2,\sqrt{2(2-\lambda)}) & Q(\lambda) &= (3,\sqrt{6(3-\lambda)})
  \end{align}
  on the Legendre elliptic curve $E_\lambda$ defined by
  $y^2=x(x-1)(x-\lambda)$. What can be said about the set of points
  $\lambda$ where both $P(\lambda)$ and $Q(\lambda)$ are torsion on
  $E_\lambda$? In \cite{mz:torsion} Masser and Zannier use the
  Pila-Zannier strategy to show that this set is finite. Here the
  analytic uniformization $\wp_\lambda(z)$ as a function of both
  variables plays the role of the uniformization, and it suffices to
  consider this function restricted to a certain compact set. Many
  other results in a similar direction have been derived using a
  similar strategy, see e.g.
  \cite{mz:torsion-jerms,mz:torsion-advances,mz:torsion-annalen,bmpz:relative-mm}.
\end{description}

\subsubsection{Questions of effectivity}

It is natural to ask to which extent, and in what instances, can an
effective form of the Pila-Zannier strategy be established. This
question is split into two parts. One problem is to effectivize the
lower bounds on sizes of Galois orbits, and the other is to
effectivize the Pila-Wilkie upper bound. Of course, in order to expect
some type of effectivity in the Pila-Wilkie theorem one must restrict
to a structure where the definable sets admit some form of effective
description.

The proof of the Manin-Mumford conjecture given in
\cite{pz:manin-mumford} relies on the orbit lower bounds of
\cite{masser:height}, which are effective. For the upper bound, the
Pila-Wilkie theorem is applied to sets defined using the uniformizing
maps of abelian varieties. In the final section of
\cite{pz:manin-mumford} Pila and Zannier hypothesize that an estimate
may be derived from an explicit description of the abelian variety and
its algebraic subvariety in terms of theta functions. Effective proofs
of the Manin-Mumford conjecture have been obtained using entirely
different methods in \cite{remond:mm,udi:mm}. We show
in~\secref{sec:elliptic}--\ref{sec:abelian} that elliptic and abelian
functions belong to the Noetherian cateogry, thus effectivizing the
upper bound and allowing an effective version of the Manin-Mumford
conjecture to be derived using the Pila-Zannier strategy.

The proof of the Andr\'e-Oort conjecture for modular curves given in
\cite{pila:andre-oort} and in subsequent works relies on lower bounds
that are generally not known to be effective (but can be made
effective assuming the Generalized Riemann Hypothesis). For the upper
bounds, the Pila-Wilkie theorem is applied to sets defined using
uniformizing maps of modular curves (and in subsequent work of Shimura
varieties). Even without an effective lower bound, an effective upper
bound could lead for example to asymptotic estimates (with an
undetermined constant) in terms of the data involved. For discussion
in this direction see \cite[Section~13]{pila:andre-oort}. An effective
version of Andr\'e's original theorem (for a product of two modular
curves) without the assumption of GRH was obtained in \cite{kuhne:ao}
(and also in \cite{bmz:ao}). We are not aware of effective results in
higher dimensions or for other Shimura curves. We show
in~\secref{sec:A_g} that the uniformizing map of the moduli space of
principally polarized abelian varieties $\cA_g$ belongs to the
Noetherian category, thus effectivizing the upper-bound for compact
subvarieties of $\cA_g$.

The proof of \cite{mz:torsion} concerning torsion anomalous points
relies on lower bounds which are effective
\cite[Section~3.4.3]{zannier:book}. For the upper bound, the
Pila-Wilkie theorem is applied to sets defined using the uniformizing
maps of elliptic families, which can be described for instance using
the Weierstrass function $\wp(z;\tau)$ as a function of both
parameters (or equivalently in terms of theta function of both
parameters). The same is probably true for many of the subsequent
works relying on the same strategy, although we have not checked the
details in every instance. Some effective results in this direction
have been obtained in \cite{hjm:six}, including through the
effectivization of the Pila-Wilkie theorem for some specific curves.
We show in~\ref{sec:theta} that $\wp(z;\tau)$ is a Noetherian
functions of both its variables, thus effectivizing the upper
bound in this context.

\subsubsection{Effectivity through differential equations}

The arithmetic-geometry applications of the Pila-Wilkie theorem involve
the use of classical functions such as exponential, elliptic and
abelian functions (for Manin-Mumford); modular functions and universal
covers of more general Shimura varieties (Andr\'e-Oort); and theta
functions (torsion anomality in families). An effective version of the
Pila-Wilkie theorem that unifies the treatment of these various
applications would have to start with a framework allowing a uniform
and effective description of each of these functions. It is natural to
look at differential equations as a possible way of describing such
functions explicitly.

The effective study of the quantitative geometry of transcendental
functions through differential equations is of course not new.
Khovanskii's theory of Pfaffian functions \cite{khovanskii:fewnomials}
provides a very successful example of this sort. Pfaffian functions
are defined in a manner similar to the Noetherian functions, but with
an extra assumption of triangularity in the
system~\eqref{eq:noetherian-sys}. With this extra assumption, general
estimates on the geometric complexity of ``Pfaffian sets'' have been
established in \cite{khovanskii:fewnomials} (see also
\cite{gv:compact-approx,gv:complexity} for additional developments).
This theory has been utilized for deriving effective versions of the
Pila-Wilkie theorem for Pfaffian curves \cite{pila:pfaff} and for
certain Pffafian surfaces \cite{jones-thomas}. These works also prove
a stronger form of the Pila-Wilkie theorem, improving the asymptotic
from sub-polynomial to polylogarithmic.

In \cite{me:rest-wilkie} an effective form of the Pila-Wilkie theorem
(and its strengthening to polylogarithmic asymptotics) was established
for definable sets of arbitrary dimension in $\R^\RE$, the structure
generated by the restricted exponential and sine functions. The proof
relies on a combination of Pfaffian methods and complex geometry, and
it appears likely that the method of proof would extend to allow
(compact restrictions of) elliptic and abelian functions. This puts
the Manin-Mumford conjecture in arbitrary dimension within the scope
of this method. On the other hand, modular functions, theta functions
and other functions required in the applications of the Pila-Wilkie
theorem to arithmetic geometry do not appear to be Pfaffian (at least
not to our knowledge), and a fundamentally different approach seems to
be required for handling them.

It seems a-priori likely (and in fact this is verified for many cases
in~\secref{sec:examples}) that the functions needed in applications to
arithmetic geometry would fall within the framework of Noetherian
functions. However, the quantitative geometric theory of Noetherian
functions is far less developed than that of the Pfaffian functions.
Khovanskii has conjectured that some ``local'' form of his theory of
Pfaffian functions should hold for Noetherian functions, and most work
has been focused on problems of a local nature
\cite{GabKho,me:ntr-dim2,me:ntr-deformations} (and even in this local
setting the conjecture is not yet fully settled). Barring very
significant progress on the general theory of Noetherian functions it
seems unlikely that the proof strategy of \cite{me:rest-wilkie} could
be carried over to the Noetherian category. For instance, we currently
do not know how to effectively bound the number of solutions of a
system of two Noetherian functions in two variables in terms of the
Noetherian parameters. In the following section we explain how it is
still possible to obtain effective estimates for the seemingly much
more complicated Pila-Wilkie theorem in arbitrary dimension without
ever addressing this very basic question.

\subsection{Sketch of the proof}

Let $\{F_j\}$ be the set of Noetherian functions defining our set
$X$ (we suppose for simplicity that only equalities are used). Since
real Noetherian functions remain Noetherian in the complex domain as
explained in~\secref{sec:complex-noetherian}, there is no harm in
viewing $\{F_j\}$ as complex analytic functions and replacing $X$
by their common zero locus in the complex domain. We note that this
simple step which is almost automatic in the Noetherian category is
unavailable in the essentially real Pfaffian category. Even if one
only wishes to prove the effective Pila-Wilkie theorem for restricted
Pfaffian functions, our proof goes through their complex continuations
which are Noetherian but no longer Pfaffian. We do not know of any
simpler method for treating the Pfaffian case in arbitrary dimension.

\subsubsection{The basic inductive step}

The proof follows an induction over dimension which is similar to the
one used in \cite{pila-wilkie}. However, for our purposes the
book-keeping needs to be done a little differently. To simplify the
presentation we assume that $X^\alg=\emptyset$. Our inductive step can
then be formulated as follows.

\begin{Prop*}[cf. Proposition~\ref{prop:Valpha-step}]
  Let $W\subset\C^n$ be an irreducible algebraic variety and suppose
  that $X\subset W$. Then there exist $NH^\e$ hypersurfaces
  $\cH_\alpha$ of degree $d$ (both $N$ and $d$ explicitly estimated in
  terms of the Noetherian parameters and $\deg W$) such that none of
  the $\cH_\alpha$ contain $W$ and $X(\Q,H)$ is contained in their
  union.
\end{Prop*}

Having established this proposition, one can start with $W=\C^n$ and
at the inductive step replace $W$ by its intersection with each of the
hypersurfaces $\cH_\alpha$ (more precisely each irreducible
component), and replace $X$ by $X\cap W$. In this way one eventually
obtains (when $\dim W=0$) a collection of $O(H^\e)$ points containing
$X(\Q,H)$.

\subsubsection{The construction of the hypersurfaces}

The construction of the hypersurfaces $\cH_\alpha$ follows a
complex-analytic strategy based on the notion of a \emph{Weierstrass
  polydisc} developed in \cite{me:interpolation,me:rest-wilkie}.
Recall that a Weierstrass polydisc for a pure-dimensional analytic set
$Y\subset\Omega$ is a polydisc $\Delta:=\Delta_z\times\Delta_w$ such
that $\dim\Delta_z=\dim Y$ and
$Y\cap(\Delta_z\times\partial\Delta_w)=\emptyset$. Under these
assumptions $Y\cap\Delta$ is a ramified cover of $\Delta_z$ of some
degree $e(Y,\Delta)$. If we denote by $\Delta^\rho$ the ``shrinking''
of $\Delta$ about its center by a factor of $\rho$, then by the method
of \cite{me:interpolation} (analogous to \cite{bombieri-pila}) one can
construct a hypersurface $\cH$ of degree $O(e(Y,\Delta))$ containing
$(Y\cap\Delta^{H^\e})(\Q,H)$.

A straightforward strategy would be to cover $X$ by sets of the form
$\Delta_\alpha^{H^\e}$ where $\Delta_\alpha$ is a Weierstrass polydisc
for $X$, with their number explicitly estimated in terms of the
Noetherian parameters. This is similar to the strategy of
\cite{me:rest-wilkie}. However, the proof in \cite{me:rest-wilkie}
relies on essentially real ideas for the construction of the
Weierstrass polydiscs (entropy estimates, Vitushkin's formula) and
requires estimates which are not available in the Noetherian category.
We proceed to explain how this is replaced in the present paper by
complex-analytic considerations.

\subsubsection{The codimension one case}

Consider the first step where $W=\C^n$. Choose one of the Noetherian
functions defining $X$ which is not identically zero, say $F$. A key
observation is that to prove the inductive step, it suffices to
replace $X$ in the statement by the larger set $Y:=\{F=0\}$. Rather
than covering $X$ by Weierstrass polydiscs, we will cover $Y$. More
specifically, we will show that for every point $p\in\Omega$ one can
construct a Weierstrass polydisc for $Y$ centered at $p$ whose size is
bounded from below in terms of the Noetherian parameters. Suppose for
simplicity that $p$ is the origin.

We study the problem of constructing Weierstrass polydiscs for
a holomorphic hypersurface $\{F=0\}$. After rescaling we may suppose
that $F$ is defined on the unit disc and has maximum norm $1$ there.
Suppose that we can find a complex line $L_w$ through the origin, and
a circle $S$ of radius $r$ around the origin in $L$ such that,
\begin{enumerate}
\item $r$ is bounded from below in terms of the Noetherian parameters,
\item $|F(w)|$ for $w\in S$ is bounded from below by some quantity
  $\delta$ depending on the Noetherian parameters.
\end{enumerate}
Then a simple argument shows that $\Delta_z\times\{|w|\le r\}$ is a
Weierstrass polydisc for $\{F=0\}$, where $\Delta_z$ is a polydisc of
radius $\sim\delta$ orthogonal to $L$.

To construct $S$ as above we use the notion of the \emph{Bernstein
  index} (see
Definitions~\ref{def:bernstein}~and~\ref{def:bernstein-Cn}). For the
purposes of this introduction, if $F$ is a function of one complex
variable in a disc $D$ then one may consider its Bernstein index given
by $\log(M/m)$ where $M$ is the maximum of $F$ on $D$ and $m$ is the
maximum on the 2-shrinking $D^2$. For functions of several variables
we take the maximum over all complex lines through the origin of the
Bernstein index of the restriction. In
Proposition~\ref{prop:R-weierstrass} we reduce the problem of finding
the circle $S$ above to the estimation of the Bernstein index.

Much is known about the estimation of Bernstein indices for solutions
of scalar linear differential equations (of any order) due to work of
\cite{iy:real-zeros,ny:london}. Moreover, these results can be
extended to solutions of polynomial (non-linear) differential
equations by using the methods of \cite{ny:chains}. A combination of
these tools suffices for producing an estimate for the Bernstein index
of a Noetherian function in terms of the Noetherian parameters (see
Theorem~\ref{thm:nonlinear-bernstein}), and allows us to finish the
proof for this case.

\subsubsection{Higher codimensions}

We now discuss the inductive step of the proof with $W$ of arbitrary
dimension. We again note that we may as well replace $X$ by the larger
set $Y:=W\cap\{F=0\}$, where $F$ is one of the Noetherian functions
defining $X$ which is not identically vanishing on $W$. We will again
seek to construct a Weierstrass polydisc for $Y$ around the origin.

We start by constructing a Weierstrass polydisc
$\Delta=\Delta_z\times\Delta_w$ for $W$ (see
Theorem~\ref{thm:algebraic-weierstrass}): since $W$ is algebraic
$\Delta$ can be constructed using the methods of
\cite{me:rest-wilkie}. Then $W$ is a ramified cover of $\Delta_z$, and
the projection $Y_F$ of $Y$ to $\Delta_z$ is given by the zero locus
of an analytic function $\cR_F$ which we call an \emph{analytic
  resultant} ~\eqref{eq:analytic-resultant}: for $z\in\Delta_z$ it is
given by the product of $F(z,w)$ over the different branches of $W$
over $z$.

If we can construct a Weierstrass polydisc $\Delta'$ for $Y_F$ in
$\Delta_z$ then a simple topological argument shows that
$\Delta'\times\Delta_w$ is a Weierstrass polydisc for $Y$ (see
Lemma~\ref{lem:X-F-weierstrass}), thus concluding our construction.
Since $Y_F$ is a hypersurface in $\Delta_z$ we are essentially reduced
back to the situation already considered, except that now we must
estimate the Bernstein index of the analytic resultant $\cR_F$ instead of
$F$ itself. This translates to choosing a complex line $L$ in
$\Delta_z$, and then studying the restriction of $F$ to the algebraic
curve $C$ obtained by lifting $L$ through the ramified cover back to
$W$.

A natural approach for studying the restriction of $F$ to an algebraic
curve $C$ is to parameterize $C$ using a map $\gamma:\C\to C$ which
itself satisfies a differential equation, and then replace $F\rest C$
by $F\circ\gamma$, now defined on $\C$ and satisfying a system of
auxiliary differential equations obtained by composing the equations
of $F$ and $\gamma$. There are two primary obstacles to this idea:
first, the curve $C$ need not be smooth, and one must somehow handle
the singular points; and second, even if $C$ is smooth, it is not
clear how to write a differential equation for the parameterization of
$C$ whose Noetherian parameters depend only of the degree of $C$. One
natural option, for instance for a plane curve $C=\{P(x,y)=0\}$, would
be to parameterize $C$ as a trajectory of the Hamiltonian field
$P_y\pd{}x-P_x\pd{}y$. However, this produces Noetherian sizes tending
to infinity along degenerating families of algebraic curves such as
$C_\e:=\{y^2+\e x=0\}$.

To overcome this problem we appeal to the theory of linear scalar
differential equations. More specifically, for every algebraic
function $y(x)$ of degree $d$ one can construct a scalar differential
operator
\begin{equation}
  L = a_0(t)\partial_t^k+\cdots+a_k(t)y, \qquad a_0,\ldots,a_k\in\C[t], \quad a_0\not\equiv0
\end{equation}
satisfying $Ly=0$. Moreover we show that the \emph{slope}
\begin{equation}
  \angle L := \max_{i=1,\ldots,k} \frac{\norm{a_i}_\infty}{\norm{a_0}_\infty}
\end{equation}
can be uniformly bounded in terms of $d$. This is a consequence of a
much more general phenomenon of ``uniform boundedness of slopes in
regular families'' which was discovered in the work of Grigoriev
\cite{grigoriev:thesis,yakovenko:grigoriev}. For instance, for the
family $C_\e$ above the operator $Ly=x\partial_xy-\tfrac12y$ provides
a differential equation for $y(x)$, uniform in $\e$. The boundedness
of the slope translates into a bound on the Noetherian size of the
auxiliary system constructed form $F\circ\gamma$. We then use various
analytic properties of the Bernstein index (notably subadditivity
under products, see Lemma~\ref{lem:bernstein-subadd}) to deduce
estimates for the Bernstein indices of $\cR_F$ and finish the proof.

\begin{Rem}
  We remark that the non-degenerating differential equations for
  parameterizations of algebraic curves are in some abstract sense
  similar to the $C^r$-parameterizaitions of Yomdin-Gromov
  \cite{yomdin:gy,yomdin:entropy,gromov:gy} that are used (in a form
  generalized to o-minimal structures) in the original work of
  Pila-Wilkie \cite{pila-wilkie}. They are used in our proof to avoid
  the same type of problem.
\end{Rem}

\subsection{Contents of this paper}

In~\secref{sec:noetherian-class} we introduce the notion of complex
Noetherian functions; prove closure of the Noetherian functions under
various operations including multiplicative inverse, composition and
compositional inverse, and formating of implicit functions; discuss
Noetherian systems with poles and Noetherian systems over smooth
varieties in place of $\C^n$; and prove a statement about the
behavior of a Noetherian functions near the boundary of $\Omega$.

In~\secref{sec:examples} we develop a large number of examples of
Noetherian functions: Klein's $j$-invariant and other modular
functions; universal covers of compact Riemann surfaces; elliptic and
abelian functions, Jacobi theta functions and $\wp(z;\tau)$ (with
respect to both variables); periods over algebraic integrals over
smooth families; and for the uniformizing map of the Siegel modular
variety $\cA_g$.

In~\secref{sec:weierstrass-analytic} we develop the general analytic
theory of Weierstrass polydiscs: we define the Bernstein index an
recall its basic properties; show how estimates on the Bernstein index
of a function $F$ can be used to effectively construct a Weierstrass
polydisc for its zero locus; introduce the notion of an analytic
resultant, and show how it can be used to inductively construct a
Weierstrass polydisc for a set $X\cap\{F=0\}$ from the Weierstrass
polydisc of $X$.

In~\secref{sec:weierstrass-noetherian} we introduce the relevant
background information on linear differential equations for algebraic
functions and the boundedness of their slope; state an estimate for
the Bernstein index of a function satisfying a polynomial (non-linear)
differential equation in terms of the Noetherian parameters; study the
restriction of a Noetherian function to an algebraic curve by
parameterization using linear differential equations; and prove the
key estimates on the Bernstein indices of analytic resultants of
Noetherian functions with respect to algebraic curves.

Finally in~\secref{sec:rational-points} we recall the relation between
Weierstrass polydiscs and the study of rational (and more generally
algebraic) points on an analytic set; prove a complex-analytic analogs
of Theorem~\ref{thm:main-k} for complex Noetherian varieties; and
reduce the general case of Theorem~\ref{thm:main-k} to its complex
version.

\section{The Noetherian class}
\label{sec:noetherian-class}

In this section we develop some elementary properties of the class of
Noetherian functions. We begin by introducing the complex analog of
the real Noetherian functions, which will be the main class considered
throughout the paper.

\subsection{Complex Noetherian functions}
\label{sec:complex-noetherian}

Let $\Omega\subset\C^n$ be a bounded domain. Then a system of the
form~\eqref{eq:noetherian-sys} where $P_{ij}$ are now allowed to be
complex polynomials is called a (complex) Noetherian system; its
holomorphic solution $\vphi:\Omega\to\C^\ell$ is called a (complex)
Noetherian chain; a function of the form $P(\vx,\vphi)$ where
$P\in\C[\vx,\vy]$ is called a (complex) Noetherian function; and the
common zero locus in $\Omega$ of a collection of Noetherian functions
is called a (complex) Noetherian variety. We use the convention that,
unless the prefix real- is explicitly used, all Noetherian constructs
are assumed to be complex.

A real Noetherian system~\eqref{eq:noetherian-sys} may be viewed as a
complex Noetherian system. Any Noetherian chain
$\vphi:\Omega_\R\to\R^\ell$ extends as a holomorphic function to some
complex domain $\Omega_\R\subset\Omega\subset\C^n$. In fact, the size
of the domain to which this complex continuation is possible can be
explicitly estimated from below in terms of the Noetherian parameters,
see Lemma~\ref{lem:boundary}. There is therefore little harm in
considering a real Noetherian chain as the restriction to the reals of
a complex Noetherian chain.

Conversely, under the identification of $\C^n$ with $\R^{2n}$ every
complex Noetherian $\vphi$ chain of dimension $n$ and length $\ell$
becomes a real Noetherian chain $\vphi_\R=(\Re\vphi,\Im\vphi)$ of dimension $2n$ and
length $2\ell$. Indeed, the system~\eqref{eq:noetherian-sys} provides
derivation rules for $\vphi_\R$ with respect to $\pd{}{z_j}$, whereas
the Cauchy-Riemann equations provides derivations rules with respect
to $\pd{}{\bar z_j}$. We leave the detailed derivation to the reader.

In conclusion, we see that in the Noetherian category the real and
complex settings are in some sense mutually-reducible. In the present
paper we will employ essentially complex arguments to the study of
Noetherian functions and assume unless otherwise stated that all
Noetherian function are complex. The equivalence above will imply that
this causes no loss of generality. We remark that this situation
stands in stark contrast to the theory of Pfaffian functions, which is
an essentially real theory: the holomorphic continuation of a
Pfaffian function defined on $\R^n$ need not itself be Pfaffian when
considered as function on $\C^n\simeq\R^{2n}$. The key difference is
that the Cauchy-Riemann equations, while algebraic, do not satisfy the
triangularity condition required of Pfaffian chains. It is this added
generality of the Noetherian class that allows our complex-analytic
treatment to go through in full generality.

\subsection{Closure properties}

We begin by noting that a union of Noetherian chains is itself Noetherian.

\begin{Lem}\label{lem:clo-union}
  Let $\vphi,\tilde\vphi:\Omega\to\C$ be two Noetherian chains of
  complexity $(n,\ell,\alpha)$ and $(n,\tilde\ell,\tilde\alpha)$
  respectively. Then $(\vphi,\tilde\vphi)$ is a Noetherian chain of
  complexity $(n,\ell+\tilde\ell,\max(\alpha,\tilde\alpha))$. Moreover,
  $\NS(\vphi,\tilde\vphi)$ is the maximum of $\NS(\vphi)$ and
  $\NS(\tilde\vphi)$.
\end{Lem}

The Noetherian class is clearly closed under differentiation.

\begin{Lem}\label{lem:clo-derivative}
  Let $\vphi:\Omega\to\C$ be a Noetherian chain with complexity
  $(n,\ell,\alpha)$ and let $F:\Omega\to\C$ be a Noetherian function
  of degree $\beta$ over $\vphi$. Then for $j=1,\ldots,n$ the derivative
  $\pd{F}{x_j}$ is a Noetherian function of degree $\beta+\alpha-1$.
\end{Lem}

\subsubsection{Closure under arithmetic operations}

Next we summarize closure properties under the basic arithmetic
operations. In light of Lemma~\ref{lem:clo-union} there is no harm
in assuming all functions involved share one Noetherian chain $\vphi$
of complexity $(n,\ell,\alpha)$.

\begin{Lem}\label{lem:clo-sum-prod}
  Let $F_1,F_2$ be two Noetherian functions over the chain $\vphi$
  with degrees $\beta_1,\beta_2$. Then $F+G$ (resp. $F\cdot G$) is a
  Noetherian function over $\vphi$ with degree $\max(\beta_1,\beta_2)$
  (resp. $\beta_1+\beta_2$).
\end{Lem}

\begin{Lem}\label{lem:clo-inv}
  Let $F=P(\vx,\vphi):\Omega\to\C$ be a Noetherian function of degree
  $\beta$, and suppose that $|F(\vx)|\ge\e>0$ for all $\vx\in\Omega$.
  Then $1/F$ is a Noetherian function of degree $1$ with respect to a
  Noetherian chain $\tilde\vphi$ with complexity
  $(n,\ell+1,\alpha+\beta+1)$ and $\NS(\tilde\vphi)$ is explicitly
  computable from $1/\e$ and the Noetherian parameters.

  A similar statement holds for the inverse of a matrix of Noetherian
  functions, where $\e$ is now a lower bound for the modulus of the
  determinant.
\end{Lem}
\begin{proof}
  We let $\tilde\vphi=(\vphi,1/F)$. To make this into a Noetherian chain we
  introduce differential equations for $1/F$,
  \begin{equation}
    \pd{(1/F)}{x_j} = -(1/F)^2 \pd{F}{x_j} =
    -(1/F)^2 \big(\pd{P}{x_j}(\vx,\vphi)+\sum_{k=1}^\ell \pd{P}{\vphi_k}(\vx,\vphi) \pd{\vphi_k}{x_j} \big)
  \end{equation}
  and note that the right hand side is a polynomial in $\vx,\tilde\vphi$ of
  degree $\alpha+\beta+1$. An upper bound for $\NS(\tilde\vphi)$ follows by
  a simple estimate.

  For the second statement it suffices to write
  $A^{-1}=(\det A)^{-1}\adj A$ which reduces the claim to the first
  statement.
\end{proof}

\subsubsection{Closure under compositions and compositional inverse}

Next we consider closure under composition and compositional inverse.

\begin{Lem}\label{lem:clo-comp}
  For $i=1,2$ let $\vphi_i:\Omega_i\to\C$ be a Noetherian chain with
  complexity $(n_i,\ell_i,\alpha_i)$. Let
  $\vF=(F_1,\ldots,F_{n_2}):\Omega_1\to\Omega_2$ be a tuple of
  Noetherian functions of degree at most $\beta_1$ over $\vphi_1$, and
  $G:\Omega_2\to\C$ be a Noetherian function of degree $\beta_2$ over
  $\vphi_2$. Then $G\circ F:\Omega_1\to\C$ is a Noetherian function of
  degree $\beta_2$ over a chain $\tilde\vphi$ with complexity
  $(n_1,\ell_1+\ell_2,\max(\alpha_1+\beta_1,\alpha_2))$ and
  $\NS(\tilde\vphi)$ explicitly computable from the Noetherian parameters.
\end{Lem}
\begin{proof}
  We let $\tilde\vphi=(\vphi_1,\vphi_2\circ\vF)$. To make this into a
  Noetherian chain we use the given equations for $\vphi_1$, and the
  chain rule for the derivatives of $\vphi_2\circ\vF$, which gives a
  polynomial combination of the components of $\vphi_2\circ F$ and the
  derivatives of $\vF$, both of which are expressible as polynomials
  in $\tilde\vphi$. We leave the details for the reader.
\end{proof}

\begin{Lem}\label{lem:clo-comp-inv}
  Let $\vphi:\Omega\to\C$ be a Noetherian chain and let
  $\vF=(F_1,\ldots,F_n):\Omega\to\tilde\Omega$ be a tuple of
  Noetherian functions. Suppose that $F$ is bijective and that
  $\abs{\det \pd{\vF}{\vx}}\ge\e>0$ for all $\vx\in\Omega$. Then the
  compositional inverse $F^{-1}$ is a tuple of Noetherian functions
  over a Noetherian chain $\tilde\vphi:\tilde\Omega\to\C$, with the
  Noetherian parameters explicitly computable in terms of the
  Noetherian parameters of $\vphi,\vF$ and $\e^{-1}$.
\end{Lem}
\begin{proof}
  For simplicity of the notation we assume that the tuple $\vphi$
  contains all the coordinate functions $x_j$. Denote the variables on
  $\tilde\Omega$ by $\vy$. Then by the chain rule,
  \begin{equation}\label{eq:clo-inverse-jacobian}
    \pd{\vF^{-1}}{\vy} = (\pd{\vF}{\vx}\circ\vF^{-1})^{-1}=\big[ \det^{-1} (\pd{\vF}{\vx}) \adj \pd{\vF}{\vx} \big]\circ\vF^{-1}
  \end{equation}
  where $\adj$ denotes the adjugate matrix. Note that by Noetherianity
  of $\vF$, the right hand side of~\eqref{eq:clo-inverse-jacobian} is
  a polynomial in $(\vphi,\det^{-1} \pd{\vF}{\vx})\circ\vF^{-1}$.

  We will show that
  $\tilde\vphi:=(\vphi,\det^{-1} \pd{\vF}{\vx})\circ \vF^{-1}:\tilde\Omega\to\C^{\ell+1}$ forms a
  Noetherian chain. By our assumption on $\vphi$ it will follow in
  particular that each component of $\vF^{-1}$ is a Noetherian function
  of degree $1$ with respect to this chain. To write equations for $\vphi\circ \vF^{-1}$
  we write
  \begin{equation}
    \pd{(\vphi\circ \vF^{-1})}{\vy} = \big(\pd{\vphi}{\vx}\circ\vF^{-1}\big)\cdot\pd{\vF^{-1}}{\vy}
  \end{equation}
  and note that the first factor is a polynomial in $\tilde\vphi$ by
  Noetherianity of $\vphi$ while the second factor is a polynomial in
  $\tilde\vphi$ by the note following~\eqref{eq:clo-inverse-jacobian}.
  Finally, to write equations for
  $(\det^{-1} \pd{\vF}{\vx})\circ\vF^{-1}$ we proceed as in
  Lemma~\ref{lem:clo-inv}, noting that the derivatives of
  $(\det \pd{\vF}{\vx})\circ\vF^{-1}$ with respect to $\vy$ are
  expressible as polynomials in $\tilde\vphi$ by what was already shown
  since it is a polynomial in $\vphi\circ\vF^{-1}$.

  It is clear that the construction is entirely explicit and the
  noetherian parameters of $\tilde\vphi$ can be explicitly estimated
  from those of $\vphi$ and $\vF$ and $\e^{-1}$ (which enters into the
  size $\NS(\tilde\vphi)$ as in Lemma~\ref{lem:clo-inv}).
\end{proof}

As a simple corollary we deduce that the Noetherian class is closed
under forming implicit functions.

\begin{Cor}\label{cor:implicit}
  Let $\Omega_x\subset\C^n_{\vx},\Omega_y\subset\C^m_\vy$ be
  relatively compact domains and let $F:\Omega_x\times\Omega_y\to\C^m$
  be a tuple of Noetherian functions such that
  $\abs{\det \pd{F}{\vy}}>\e>0$ on the set $\cG=F^{-1}(0)$. If $\cG$
  is the graph of a function $G:\Omega_x\to\Omega_y$ then $G$ is
  Noetherian with parameters explicitly computable from the Noetherian
  parameters of $F$ and $\e$.
\end{Cor}
\begin{proof}
  By assumption the map $H:(\vx,\vy)\to(\vx,F)$ satisfies
  $\abs{\pd{H}{(\vx,\vy)}}>\e$ on $\cG$. We choose a neighborhood
  $\Omega\subset\Omega_\vx\times\Omega_\vy$ of $\cG$ such that $H$ is
  injective on $\Omega$: it is enough to ensure that $H(\vx,\vy)$ is
  injective in $\vy$ for fixed $\vx$, and $\Omega$ with this property
  can be chosen by the inverse mapping theorem (here we use that $\cG$
  has only one point with a given $\vx$ coordinate). By
  Lemma~\ref{lem:clo-comp-inv} the inverse $H^{-1}:F(\Omega)\to\Omega$
  is Noetherian with computable parameters as above. Finally note that
  $\Omega_\vx\times\{0\}=H(\cG)\subset H(\Omega)$ and
  $H^{-1}(\vx,0):\Omega_\vx\to\Omega$ is a Noetherian function of the
  form $H^{-1}(\vx,0)=(\vx,G(\vx))$, with $G$ satisfying the
  conditions of the statement.
\end{proof}

\subsection{Systems with poles}
\label{sec:noetherian-poles}

In our definition of a Noetherian system~\eqref{eq:noetherian-sys} we
assume that the derivatives of $\vphi$ are expressible as
\emph{polynomials} in $\vx,\vphi$. In general one may also consider
rational systems
\begin{equation}\label{eq:noetherian-sys-rational}
  \pd{\phi_i}{x_j} = \frac{Q_{i,j}(\vx,\vphi)}{R_{i,j}(\vx,\vphi)}, \qquad
  \begin{aligned}
    i=1,\ldots,\ell \\ j=1,\ldots,n
  \end{aligned}  
\end{equation}
but in this case a more careful notion of ``Noetherian size'' is
required to make our main results hold.

\begin{Ex}
  Consider the following Noetherian system in independent variable $x$
  and dependent variables $(e,f,g)$,
  \begin{align}
    \pd{e}{x}&=0 & \pd{f}{x} &= g/e & \pd{g}{x} &= -f/e.
  \end{align}
  For every fixed value of $0<\e<1$, the tuple
  $(\e,\sin(x/\e),\cos(x/\e))$ in the domain $\Omega=\{0<x<1\}$ forms
  a Noetherian chain for this system. However, the number of rational
  points of height $H$ in the set $\Omega\cap\{\sin(x/\e)=0\}$ for
  $\e=1/(\pi H)$ is $H$, so clearly our notion of Noetherian size for
  such systems must tend to infinity as $\e\to0$ if we are to expect
  a sub-polynomial asymptotic for the number of rational points, with the
  constant depending only on the Noetherian size.
\end{Ex}

In this paper we will only consider the case where the Noetherian
chain $\vphi$ remain bounded away from the polar locus. In this case
one may, by a simple reduction,
translate~\eqref{eq:noetherian-sys-rational} back into polynomial form
as follows. We introduce additional dependent variables $\rho_{i,j}$
for $\frac1{R_{i,j}}$ and recast~\eqref{eq:noetherian-sys-rational} in
the form
\begin{equation}
  \begin{aligned}
    \pd{\phi_i}{x_j} &= \rho_{i,j}Q_{i,j}(\vx,\vphi) \\
    \pd{\rho_{i,j}}{x_k} &= -\rho_{i,j}^2 \pd{R_{i,j}(\vx,\vphi)}{x_k}
  \end{aligned}
  \quad\text{for}\quad
  \begin{aligned}
    i=1,\ldots,\ell \\ j,k=1,\ldots,n
  \end{aligned}  
\end{equation}
where we express $\pd{R_{i,j}(\vx,\vphi)}{x_k}$ as a polynomial in
$\vx,\vphi,\rho_{i,j}$ using the derivation
rule~\eqref{eq:noetherian-sys-rational}, replacing all occurrences of
$\frac1{R_{i,j}}$ by $\rho_{i,j}$. It is then simple to see that for
any solution $\vphi$ of~\eqref{eq:noetherian-sys-rational},
$(\vphi,\vrho)$ forms a solution of the system above, and moreover that
the Noetherian size of this solution can be explicitly estimates in
terms of the Noetherian size of $\vphi$ and the minimum value attained
by $|R_{i,j}(\vx,\vphi)|$, i.e. the ``distance'' to the polar locus of
\eqref{eq:noetherian-sys-rational}.

\subsection{Systems over smooth algebraic varieties}
\label{sec:charts}

Let $S\subset\C^N$ be an irreducible smooth algebraic variety of
dimension $n$. One can generalize the notion of a Noetherian system on
$\C^n$ to Noetherian systems over $S$. For this purpose it is more
convenient to use differential form notation. Let $T^*S$ denote the
cotangent bundle of $S$. We denote by $\Omega^1_S$ the (algebraic)
sheaf of sections of this bundle. If $U\subset S$ is a (possibly
non-algebraic) domain then we will denote by $\Omega^1_S(U)$ the
sections admitting a regular extension to some Zariski open
neighborhood of $U$. We denote by
\begin{equation}
  \Omega^1_S[z_1,\ldots,z_k]:=\Omega^1_S\otimes\C[z_1,\ldots,z_k]
\end{equation}
the sheaf of sections of $T^*S$ depending polynomially on additional
variables $z_1,\ldots,z_k$. We will use similar notations with $\cO_S$ for
the structure sheaf of $S$.

Let $U\subset S$ be a relatively compact domain. A collection of
holomorphic functions $\vphi:=(\phi_1,\ldots,\phi_\ell):U\to\C^\ell$
is called a Noetherian chain if it satisfies a \emph{Noetherian system
  over $S$},
\begin{equation}\label{eq:noetherian-S-sys}
  \d\vphi = \omega, \qquad \omega\in\C^\ell\otimes\Omega_S^1[\vphi](U)
\end{equation}
where $\omega$ is a vector of one-forms depending polynomially on
$\vphi$. A function from the ring $\cO_S[\vphi](U)$ is called a
Noetherian function over $S$.

Much of the material developed in this paper could be generalized to
Noetherian systems over a smooth variety. However, this would
introduce additional (mostly notational) difficulties that we prefer
to avoid in the interest of readability. Instead we show that one can
form local charts on $S$ where the system~\eqref{eq:noetherian-S-sys}
pulls back to a standard Noetherian system on a subset of $\C^n$. More
formally let $s_0\in S$. Since $S\subset\C^N$ is smooth, one can
choose $N-n$ polynomials vanishing on $S$ with linearly independent
differentials in a neighborhood of $s_0$. Let $\pi:S\to\C^n$ be a
linear projection which is submersive at $s_0$. By
Corollary~\ref{cor:implicit} applied to the collection of polynomials
above, one may construct a biholomorphic Noetherian map
$\Psi=(\psi_1,\ldots,\psi_N):U_0\to S$ for some domain
$U_0\subset\C^n$ whose image contains $s_0$, and whose inverse is
given by $\pi$. Since $\pi$ is a linear projection, we see that
$\Psi^{-1}$ is Noetherian over $S$.

Represent the Noetherian system for $\Psi$ in differential form
\begin{equation}\label{eq:Psi-sys}
  \d\Psi = \eta(\Psi), \qquad \eta\in \C^N\otimes\Omega^1_{\C^n}[\Psi](U_0).
\end{equation}
Since $S$ is smooth, among the one-forms
$\d x_1,\ldots,\d x_N\in(\C^N)^*$ there are $n$ forms whose
restrictions to $S$ form a basis for $T^*_S$ over $s_0$. Expressing
$\omega$ in terms of this basis, we see that after possibly
restricting to a Zariski open neighborhood $\tilde U$ of $s_0$ we may
write $\omega$ as the restriction to $\tilde U$ of an element of
$\C^\ell\otimes(\C^N)^*\otimes\cO_S[\vphi](\tilde U)$. Similarly,
since any function in $\cO_S(\tilde U)$ is the restriction of some
rational function on $\C^N$ with polar locus disjoint from $\tilde U$
we may take $\omega$ to be the restriction to $\tilde U$ of
$\tilde\omega\in\C^\ell\otimes(\C^N)^*\otimes\C[\vx,\vphi]_f$ where
the localization is by some polynomial $f\in\C[\vx]$ non-vanishing in
$\tilde U$. Shrinking $U_0$ we may also assume that
$\Psi(U_0)\subset\tilde U$.

We claim that the tuple $(\Psi,\vphi\circ\Psi):U_0\to\C^{N+\ell}$
containing the pullbacks of $\vphi$ by $\Psi$ is a Noetherian chain.
The equation~\eqref{eq:Psi-sys} is already in Noetherian form, and we
proceed to derive Noetherian equations for $\vphi\circ\Psi$.
From~\eqref{eq:noetherian-S-sys} and~\eqref{eq:Psi-sys} we have
\begin{equation}
  \begin{aligned}
    \d(\vphi\circ\Psi)&=(\d\vphi)_\Psi\cdot\d\Psi = \tilde\omega(\vphi\circ\Psi)\cdot\eta(\Psi) \\
     &\in \C^\ell\otimes\Omega^1_{\C^n}(U_0)\otimes\C[\vphi\circ\Psi,\Psi]_{f(\Psi)}
  \end{aligned}
\end{equation}
where we used the pairing of $\C^N$ and $(\C^N)^*$. Since $f$ is
non-vanishing on $\Psi(U_0)$ by assumption we may view this as a
Noetherian system with poles by the construction
of~\secref{sec:noetherian-poles}. Finally, any Noetherian function
$F\in\cO_S[\vphi](U)$ may be expressed as an element of
$\C[\vx,\vphi]_g$ as explained above, and its pullback
$F\circ\Psi\in\C[\Psi,\vphi\circ\Psi]_{g(\Psi)}$ is Noetherian over
$\Psi,\vphi\circ\Psi$ (with poles at the zeros of $g(\Psi)$).

\subsection{Boundary behavior}

We finish this section with a simple lemma showing that the behavior
of the Noetherian chain at in a neighborhood of the boundary of a
domain can be controlled in terms of the Noetherian size.

\begin{Lem}\label{lem:boundary}
  Let $\vphi:\Omega\to\C^n$ be a (possibly real) Noetherian chain (so
  that we allow $\Omega$ to be a bounded domain in $\R^n$ as well) and
  set $S:=\NS(\vphi)$. Then $\vphi$ extends as a complex Noetherian
  chain on the complex $\rho$-neighborhood of $\Omega$, where
  $\rho=S^{-O(1)}$ and the Noetherian size of $\vphi$ in this larger
  domain is bounded by $2S$.
\end{Lem}
\begin{proof}
  It is a simple exercise to verify using~\eqref{eq:noetherian-sys}
  that for $t_0\in\Omega$ we have
  \begin{equation}
    \abs{\pd{^\vj\phi_i}{\vx^\vj}(t_0)} \le \vj! S^{O(|\vj|)}
  \end{equation}
  where the asymptotic constants depend on $n,l,\alpha$. From this it
  follows that $\vphi$ extends holomorphically to a
  $\rho$-neighborhood of $\Omega$ with $\rho=S^{-O(1)}$, and it
  obviously continues to satisfy~\eqref{eq:noetherian-sys} in this
  larger domain. It is also clear that one may choose $\rho$ as above
  such that the Noetherian size of $\vphi$ in this larger domain is
  bounded by $2S$.
\end{proof}

\section{Examples of Noetherian functions}
\label{sec:examples}

The elementary functions $e^z,\sin z$ and $\cos z$ (restricted to any
relatively compact domain in $\C$) form classical examples of
Noetherian functions. We proceed with some less trivial examples. In
this section we shall freely use Noetherian systems with poles as
developed in~\secref{sec:noetherian-poles}.

\subsection{Klein's $j$-invariant and other modular functions}
\label{sec:j-func}

The Klein $j$-invariant is the unique function $j:\H\to\C$ which is
$\SL(2,\Z)$-invariant, holomorphic in $\H$, has a simple pole at the
cusp, and satisfies $j(e^{2\pi i/3})=0$ and $j(i)=1728$. We will
denote the coordinate on $\H$ by $\tau$ and on $\C$ by $z$.

The $j$ function realizes the identification
$\SL(2,\Z)\backslash\H\simeq \C=\C P^1\setminus\{\infty\}$, but note
that it is not a covering map: it is ramified over the points $0,1728$
corresponding to the orbits of $e^{2\pi i/3},i$ (which have a
non-trivial stabilizer in $\SL(2,\Z)$). The $j$-function is known to
satisfy a differential equation of order $3$. To define it, recall
that the \emph{Schwarzian derivative} $S_z(f)$ and \emph{automorphic
  derivative} $D_z(f)$ of a function $f$ are defined by
\begin{align}\label{eq:schwarz-auto}
  S_z(f) &:= \frac{f'''}{f'}-\frac32\left(\frac{f''}{f'}\right)^2 & D_z(f)&:=S_z(f)/(f')^2.
\end{align}
If $g\in\Aut \C P^1$ is a M\"obius transformation and $f$ is a
holomorphic function on some domain in $\C P^1$ then
\begin{align}
  S_z(g\circ f)&=S_z(f) & D_z(f\circ g)=D_z(f)\circ g
\end{align}
wherever both sides are defined. In particular, since $j$ is
automorphic with respect to $\SL(2,\Z)$ it follows that $D_\tau j$ is
as well. Since $j$ is a Hauptmodul for $\SL(2,\Z)$ it follows that
$D_\tau j$ is a rational function of $j$, and more explicitly
\cite[page~20]{masser:heights}
\begin{equation}\label{eq:j-eq}
  D_\tau j = \frac{j^2-1968j+2645208}{2j^2(j-1728)^2}.
\end{equation}
We now cast~\eqref{eq:j-eq} as a Noetherian system in the free
variable $\tau$ and dependent variables $j,j_1,j_2$,
\begin{equation}\label{eq:j-sys}
  \begin{aligned}
    \pd{j}\tau &= j_1, \pd{j_1}\tau = j_2, \\
    \pd{j_2}\tau &= \frac{3j_2^2}{2j_1} + \frac{j^2-1968j+2645208}{2j^2(j-1728)^2} j_1^3 
  \end{aligned}
\end{equation}
It is straightforward to verify using~\eqref{eq:j-eq} that $\vphi_j$
given by
\begin{align}\label{eq:j-sys-sol}
  j&=j(\tau) & j_1&=j'(\tau) & j_2&=j''(\tau)
\end{align}
forms a solution of~\eqref{eq:j-sys}. If
$\Omega\subset\H\setminus\SL(2,\Z)\cdot\{e^{2\pi i/3},i\}$ is
relatively compact then~\eqref{eq:j-sys} is defined in $\Omega$ since
$j'$ and $j(j-1728)$ only vanish on the orbits of $e^{2\pi i/3},i$.
The singularities of our equations over the orbits of $e^{2\pi i/3},i$
correspond to the fact the $j$ is ramified at these points. In fact
the methods developed in this paper may be extended to cover such
``mild'' singularities, but in the interest of clarity we avoid this
extra generality. For ``reasonable'' domains $\Omega$ it should also
be possible to estimate $\NS(\vphi_j)$ explicitly, but we do not
pursue this direction.

If instead of $\SL(2,\Z)$ one considers for instance its modular
subgroup $\Gamma(2)$ then the resulting map
$\lambda:\H\to\Gamma(2)\backslash\H\simeq\C
P^1\setminus\{0,1,\infty\}$ is a covering map. By the same arguments
as above, the $\lambda$ function satisfies a differential equation
which can be transformed into a Noetherian system. Note that in this
case, since the $\lambda$ function is a non-ramified cover, the
resulting solution $\vphi_\lambda$ is defined for any relatively compact domain
$\Omega\subset\H$.

\subsection{Universal covers of compact Riemann surfaces}
\label{sec:covering}

Let $C$ be a compact Riemann surface of genus greater than one. Then
by the uniformization theorem there is a (uniquely defined up to a
M\"obius transformation) universal covering map
$f:\H\to\Gamma\backslash\H\simeq C$ where $\Gamma\subset\PSL(2,\R)$ is
a discrete subgroup acting freely on $\H$. We fix an embedding
$C\to\C P^k$ of $C$ as an smooth algebraic curve in projective space.

Let $p\in C$, and let $\vx=(x_1,\ldots,x_k):\C P^k\to\C^k$ denote a
set of affine coordinates corresponding to an affine subspace of
$\C P^k$. We may and do assume that $p$ is in the finite part of
$\vx$, and moreover that $\d x_j\rest C$ does not vanish at $p$ for
$j=1,\ldots,k$. Set $f_j:=x_j\circ f:\H\to\C$. We will show that there
exists a neighborhood $p\in U_p\subset C$ such that $f_j$ form
Noetherian functions in any relatively compact domain
$\Omega_p\subset f^{-1}(U_p)$. By compactness of $C$ we can then study
the universal cover $f$ using finitely many such charts.

Since $f_j$ is a meromorphic and $\Gamma$-invariant, the same is true
for $D_\tau f_j$ which therefore defines a meromorphic function on
$\Gamma\backslash\H\simeq C$. By the GAGA principle (or Chow's theorem) it follows
that
\begin{equation}
  D_\tau f_j = R_j\circ f
\end{equation}
where $R_j$ is a rational function on $C$. Moreover,
using~\eqref{eq:schwarz-auto} it is easy to verify that $R_j$ is
regular at $p$: the poles of $D_\tau f_j$ only occur at poles or zeros
of $f_j'(\tau)=\d x_j(f'(\tau))$, neither of which contains $p$ by our
assumption on $\vx$. Therefore one can write $R_j=P_j/Q_j$ where
$P_j,Q_j$ are polynomials in the coordinates $\vx$ and $Q_j(p)\neq0$.

Fix a neighborhood $p\in U_p\subset C$ where $\vx$ is finite,
$f_1',\ldots,f_k'$ are finite, and $Q_1,\ldots,Q_k$ are non-zero. Then
we have a system
\begin{equation}\label{eq:covering-eq}
  D_\tau f_j = \frac{P_j(f_1,\ldots,f_k)}{Q_j(f_1,\ldots,f_k)}
\end{equation}
and moreover the $f_1,\ldots,f_k$ are bounded and the denominators
$Q_j(f_1,\ldots,f_k)$ and the derivatives $f'_j$ are bounded away from
zero in any relatively compact domain $\Omega_p\subset f^{-1}(U_p)$.
Then~\eqref{eq:covering-eq} may be viewed as a Noetherian system (with
poles) and the conditions above guarantee that $f_1,\ldots,f_k$
correspond to a holomorphic solution of this Noetherian system in
$\Omega_p$.

\subsection{Elliptic functions}
\label{sec:elliptic}

In~\secref{sec:abelian} we show that all abelian functions satisfy
Noetherian systems. However, since the case of dimension one (elliptic
functions) is of classical importance we begin with a more explicit
description of this case. Recall that the field of doubly-periodic
meromorphic functions with periods $\w_1,\w_2$ is generated by
$\wp,\wp'$ where
\begin{equation}
  \wp(z)=\wp(z;\w_1,\w_2) = \frac1{z^2}+\sum_{n^2+m^2\neq0} \left[ \frac1{(z+m\w_1+n\w_2)^2}-\frac1{(m\w_1+n\w_2)^2} \right].
\end{equation}
The $\wp$-function satisfies the differential equation
\begin{equation}
  (\wp'(z))^2 = 4\wp^3(z)-g_2\wp(z)-g_3
\end{equation}
where $g_2,g_3$ are certain invariants depending on $\w_1,\w_2$. Taking derivative,
dividing by $\wp'$ and taking another derivative we see that the $\wp$ function satisfies
the \emph{stationary Korteweg--de Vries equation}
\begin{equation}\label{eq:p-kdv}
  \wp'''(z) = 12\wp\wp',
\end{equation}
which is expressible as a Noetherian system in the dependent variables
$\wp,\wp_1,\wp_2$ as follows
\begin{align}\label{eq:elliptic-sys}
  \wp' &= \wp_1 & \wp_1'&=\wp_2 & \wp_2' &= 12\wp\wp_1.
\end{align}
We remark that the equations~\eqref{eq:elliptic-sys} do not depend on
the invariants $g_2,g_3$, meaning that the Noetherian complexity does
not diverge as the periods $\w_1,\w_2$ degenerate.

\subsection{Abelian functions}
\label{sec:abelian}

Recall that a \emph{complex torus} $X$ of dimension $g$ is a quotient
$\C^g/\Lambda$ where $\Lambda\subset\C^g$ is a lattice of rank $2g$. A
complex torus $A$ which is also a projective variety over $\C$ is said
to be an \emph{abelian variety}. A meromorphic function on an abelian
variety is called an \emph{abelian function}. We fix an abelian
variety $A$ and an embedding $A\to\C P^k$ of $A$ as an smooth
projective variety.

Let $\pi:\C^g\to\C^g/\Lambda\simeq A$ denote the quotient map, and
denote by $\vz=(z_1,\ldots,z_g)$ the coordinates on $\C^g$. Let
$\vx=(x_1,\ldots,x_k):\C P^k\to\C^k$ denote a set of affine
coordinates corresponding to an affine subspace of $\C P^k$ and set
$f_j:=x_j\circ\pi$. We will show that $f_1,\ldots,f_k$ are Noetherian
functions of the variables $\vz$ away from the polar locus. Since $A$
is compact, one can then study the map $\pi$ using finitely many
affine charts $\vx$. The proof is similar to the one given
in~\secref{sec:covering}.

For $i=1,\ldots,k$ and $j=1,\ldots,g$, the derivative $\pd{f_i}{z_j}$
is $\Lambda$-invariant (since $f_j$ is), and it therefore defines a
meromorphic function on $\C^g/\Lambda\simeq A$. By the GAGA principle
(or Chow's theorem) it follows that
\begin{equation}
  \pd{f_i}{z_j} = R_i\circ\pi
\end{equation}
where $R_i$ is a rational function on $A$. Moreover, $R_i$ is regular
on the finite part of the chart $\vx$, and can therefore be expressed
as a polynomial $P$ in the $\vx$ variables. Thus we have a system of
Noetherian equations
\begin{equation}
  \pd{f_i}{z_j} = P_i(f_1,\ldots,f_k).
\end{equation}

\subsection{Jacobi theta functions, thetanulls and $\wp(z;\tau)$}
\label{sec:theta}

The theta functions are a subject of a large number of inconsistent
notational variations. We stick here to the conventions employed by
\cite{ww:modern-analysis,as:handbook}. We remark that in these sources
the notation for the elliptic function $\wp(z;\w_1,\w_2)$ varies
slightly from the one used previously in~\secref{sec:elliptic}: it
denotes a function with periods $2\w_1,2\w_2$ rather than $\w_1,\w_2$.
To avoid further complicating the references, we use this alternative
normalization for this section. We write $\wp(z;\tau):=\wp(z;1,\tau)$.

The Jacobi theta function is the holomorphic function given by
\cite[21.1]{ww:modern-analysis}
\begin{equation}
  \vartheta:\C\times\H\to\C, \qquad \vartheta(z;\tau) = \sum_{n\in\Z} (-1)^ne^{\pi i n^2\tau +2i n z}.
\end{equation}
We also write $\vartheta_4$ for $\vartheta$. Three additional variants are \cite[21.11]{ww:modern-analysis}
\begin{equation}
  \begin{aligned}
    \vartheta_3(z;\tau)&:=\vartheta_4(z+\tfrac12\pi;\tau), \\
    \vartheta_1(z;\tau)&:=-ie^{iz+\tfrac14\pi i\tau}\vartheta_4(z+\tfrac12\pi\tau;\tau), \\
    \vartheta_2(z;\tau)&:=\vartheta_1(z+\tfrac12\pi;\tau).
  \end{aligned}
\end{equation}
It is common to omit $\tau$ from the notation, writing $\vartheta(z)$
for $\vartheta(z;\tau)$. The theta functions
$\vartheta_1,\vartheta_2,\vartheta_3,\vartheta_4$ vanish at $(z;\tau)$
if an only if $z$ is congruent modulo $\pi,\pi\tau$ to
$0,\tfrac\pi2,\tfrac\pi2(1+\tau),\tfrac\pi2\tau$ respectively
\cite[21.12]{ww:modern-analysis}. Each of the theta functions
satisfy a variant of the heat equation \cite[21.4]{ww:modern-analysis}
\begin{equation}\label{eq:theta-heat}
  \delta\vartheta_j = -\tfrac14\partial_z^2\vartheta_j
\end{equation}
where $\delta:=\frac1{\pi i}\pd{}{\tau}$ and $\partial_z:=\pd{}z$.

We now restrict attention to the \emph{thetanulls}, i.e. the functions
$\vartheta_j(0):=\vartheta_j(0;\tau)$ for $j=2,3,4$ (note
$\vartheta_1(0)\equiv0$). From the description of the zeros of the
theta functions above, we see that the thetanulls are nowhere
vanishing on $\H$. The logarithmic derivatives
$\psi_j:=\delta\vartheta_j(0)/\vartheta_j(0)$ satisfy a system of
non-linear differential equations due to Halphen \cite{halphen:system},
\begin{equation}\label{eq:halphen}
  \begin{aligned}
    \delta\psi_2 &= 2(\psi_2\psi_3+\psi_2\psi_4-\psi_3\psi_4),\\
    \delta\psi_3 &= 2(\psi_2\psi_3+\psi_3\psi_4-\psi_2\psi_4),\\
    \delta\psi_4 &= 2(\psi_2\psi_4+\psi_3\psi_4-\psi_2\psi_3).
  \end{aligned}
\end{equation}
It is easy to recast~\eqref{eq:halphen} as a rational Noetherian
system for the thetanulls $\vartheta_j(0)$ and their first derivatives
$\delta\vartheta_j(0)$, for example
\begin{equation}
  \delta^2\vartheta_2(0) = 2\vartheta_2(0)(\psi_2\psi_3+\psi_2\psi_4-\psi_3\psi_4)+(\delta\vartheta_2)^2/\vartheta_2(0)
\end{equation}
and similarly for $j=3,4$. Since the denominators of these system are
given by the thetanulls which are nowhere vanishing in $\H$, we
conclude that the thetanulls are Noetherian in every relatively
compact subdomain of $\H$.

We now return to the general study of the theta functions and their
relation to the Weierstrass elliptic functions. Recall that the
Weierstrass $\zeta$ function is defined by the conditions
\begin{equation}
  \partial_z\zeta(z;\tau) = -\wp(z;\tau), \qquad \lim_{z\to0} \zeta(z;\tau)-z^{-1}=0.
\end{equation}
From~\eqref{eq:p-kdv} it follows that $\zeta$ satisfies
\begin{equation}\label{eq:zeta-eq}
  \partial_z^4\zeta = -12(\partial_z\zeta)(\partial_z^2\zeta).
\end{equation}
The $\zeta$-function can be expressed in terms of theta functions
as follows \cite[18.10.7, 18.10.18]{as:handbook}
\begin{equation}\label{eq:zeta-vs-theta}
  \zeta(z;\tau) = \eta z+\partial_z(\log\vartheta_1(\frac{\pi z}2;\tau)), \qquad
  \eta := -\frac{\pi^2}{12} \frac{\vartheta_1'''(0)}{\vartheta_1'(0)}.
\end{equation}
We claim that $\eta$ is a Noetherian function with respect to the
system constructed above for the thetanulls. To see this recall
that the thetanulls satisfy the fundamental relation
\cite[21.41]{ww:modern-analysis}
\begin{equation}
  \vartheta_1'(0) = \vartheta_2(0)\vartheta_3(0)\vartheta_4(0).
\end{equation}
Combining with~\eqref{eq:theta-heat} this gives
\begin{equation}
  \eta = \frac{\pi^2}{3} \frac{\delta(\vartheta_2(0)\vartheta_3(0)\vartheta_4(0))}{\vartheta_2(0)\vartheta_3(0)\vartheta_4(0)}
\end{equation}
and the $\delta$-derivative can be expressed in terms of the
Noetherian derivation rules.

We now construct a Noetherian system in the independent variables
$z,\tau$ including the functions $\zeta(z;\tau)$ and
$\tilde\vartheta_1(z,\tau):=\vartheta_1(\pi z/2;\tau)$. We start with
the system constructed above for the thetanulls (by definition
independent of the variable $z$), and add five additional dependent
variables $\vartheta_1$ and $\zeta=\zeta_0,\zeta_1,\zeta_2,\zeta_3$
for $\zeta$ and its first three $\partial_z$-derivatives. For the $z$
derivatives we have the following derivation rules
\begin{equation}\label{eq:zeta-sys-z}
  \begin{aligned}
    \partial_z\zeta_j&=\zeta_{j+1}, \qquad j=0,1,2, \\
    \partial_z\zeta_3&=-12\zeta_1\zeta_2, \\
    \partial_z\tilde\vartheta_1&=(\zeta-\eta z)\tilde\vartheta_1
  \end{aligned}
\end{equation}
where the last two equations follow from~\eqref{eq:zeta-eq}
and~\eqref{eq:zeta-vs-theta} respectively. For the $\tau$ derivatives
we have
\begin{equation}\label{eq:zeta-sys-tau}
  \begin{aligned}
    \partial_\tau\tilde\vartheta_1 &= -\frac{i}\pi \partial_z^2\tilde\vartheta_1, \\
    \partial_\tau\zeta &= \partial_\tau(\eta z)+\partial_\tau \frac{\partial_z\tilde\vartheta_1}{\tilde\vartheta_1} \\
    \partial_\tau\zeta_j &= \partial_z^j(\partial_\tau\zeta), \qquad j=1,2,3,
  \end{aligned}
\end{equation}
where the first two equations follows from~\eqref{eq:theta-heat}
and~\eqref{eq:zeta-vs-theta} respectively. The
system~\eqref{eq:zeta-sys-tau} can be rewritten as a Noetherian
system: for the first equation, one can rewrite
$\partial_z^2\tilde\vartheta_1$ as a polynomial using the derivation
rules of~\eqref{eq:zeta-sys-z}; for the second equation one can
proceed similarly using additionally the first equation for
$\partial_\tau\tilde\vartheta_1$; and for the final three equations
one proceeds similarly using the second equation for
$\partial_\tau\zeta$. We leave the detailed derivation for the reader.

In conclusion, we have obtained a rational Noetherian system for
$\tilde\vartheta_1(z;\tau)$ and $\zeta(z;\tau)$ with the polar locus
given by $\{\tilde\vartheta_1=0\}$. This polar locus is given by all
pairs $(z;\tau)$ such that $z$ is congruent to $0$ modulo $2,2\tau$,
which is the same as the polar locus of $\zeta(z;\tau)$. Since
$\wp(z;\tau)=-\partial_z\zeta(z;\tau)$, we see that $\wp(z;\tau)$ is
also Noetherian outside its polar locus.

\subsection{Periods of algebraic integrals, Gauss-Manin connection}
\label{sec:periods}

We follow \cite{griffiths:periods} for basic facts concerning the
Gauss-Manin connection. Let $X,S$ be smooth algebraic varieties over
$\C$, and $f:X\to S$ be a rational proper holomorphic map such that
$\d f$ is everywhere of maximal rank. For $s\in S$ we denote
$V_s:=f^{-1}(s)$, and think of $\{V_s\}_{s\in S}$ as a family of
smooth, complete varieties depending on the parameter $s$. There is an
algebraic vector bundle $\Hdr^1(X/S)\to S$, called \emph{relative
  cohomology of $X$ over $S$}, whose fiber over a point $s\in S$ is
given by $H^1(X_s,\C)$. This vector bundle is equipped with a
canonical flat connection $\nabla$ called the \emph{Gauss-Manin}
connection,
\begin{equation}
  \nabla:\Hdr^1(X/S)\to\Omega^1_S\otimes\Hdr^1(X/S)
\end{equation}
where $\Omega^1_S$ denotes regular 1-forms on $S$. Geometrically, if
$\delta\in H_1(X_{s_0},\Z)$ and $\delta(s)$ denotes the continuation of
$\delta$ for $s$ near $s_0$ using Ehreshmann's lemma, then $\delta(s)$
are flat sections of the dual connection $\nabla^*$ on the relative
homology bundle $H_1(X/S)$.

Fix some point $s_0\in S$ and let
$\Phi_1,\ldots,\Phi_m:S\to\Hdr^1(X/S)$ be rational sections which form
a basis at $s_0$. Choose any basis
$\delta_1(s_0),\ldots,\delta_m(s_0)$ of $H_1(X_{s_0},\Q)$ and extend
it to sections $\delta_1,\ldots,\delta_m:S\to H_1(X/S)$ as above. If
we denote by $(\cdot,\cdot)$ the pairing between homology and
cohomology then we have the \emph{period matrix},
\begin{equation}\label{eq:period-matrix}
  X(s) = \begin{pmatrix}
    (\Phi_1(s),\delta_1(s)) &\cdots& (\Phi_1(s),\delta_m(s)) \\
    & \ddots & \\
    (\Phi_m(s),\delta_1(s)) &\cdots& (\Phi_m(s),\delta_m(s))
  \end{pmatrix}.
\end{equation}
Since $\delta_j$ are flat with respect to $\nabla^*$ we have
$\d(\Phi_i,\delta_j)=(\nabla(\Phi_i),\delta_j)$ and in matrix form we
have a linear differential equation
\begin{equation}\label{eq:GM-matrix}
  \d X = A \cdot X
\end{equation}
where $A$ is an $m\times m$ matrix whose coefficients are rational
one-forms in $S$ and regular whenever $\Phi_1,\ldots,\Phi_m$ form a
basis.

If for instance $S=\C^n$ then~\eqref{eq:GM-matrix} restricted to each
of the columns of $X(s)$ is a (rational) Noetherian system which is
regular wherever $A$ is (and in particular at $s_0$). For a general
variety $S$, the system~\eqref{eq:GM-matrix} is a Noetherian system
over $S$ in the sense of~\secref{sec:charts}. In conclusion, we see
that the entries of the period matrix are Noetherian functions in a
neighborhood of any point $s_0\in S$.

\subsection{The moduli space of principally polarized abelian varieties}
\label{sec:A_g}

Periods play a central role in the construction of moduli spaces of
principally polarized abelian varieties (with additional structure).
The first classical example is that of elliptic curves. If $E$ is an
elliptic curve, $\omega$ a holomorphic one-form on $E$ and
$\delta_1,\delta_2\in H_1(E,\Z)$ two cycles with intersection number
$1$, then the upper half of the period matrix~\eqref{eq:period-matrix}
consists of the two elliptic integrals
\begin{align}
  I_1&:=\oint_{\delta_1}\omega & I_2&:=\oint_{\delta_2}\omega
\end{align}
In the universal cover $\C\to\C/\Lambda\simeq E$ the form $\omega$
corresponds (up to scalar) to $\d z$ and the two periods correspond to
two generators of the lattice $\Lambda$. The ratio of the two periods
is one of the points in the upper half-space $\H$ representing $E$.

If we consider a smooth family of elliptic curves $\{E_s\}_{s\in S}$
then the periods $I_1,I_2$ continue as maps $I_1(s),I_2(s)$ as
in~\secref{sec:periods}. In a neighborhood of any point $s_0\in S$
these two functions are Noetherian, and we may assume without loss of
generality that $I_2(s)\neq0$ in the neighborhood so that the function
$I_1(s)/I_2(s)$ mapping each $s$ to its representative in $\H$ is also
Noetherian. For instance if we let
\begin{equation}
  \Gamma(N) := \left\{
      \begin{pmatrix}
        a&b\\c&d
      \end{pmatrix}: a\equiv d\equiv\pm1(\text{mod}\ N),\quad b\equiv c\equiv0(\text{mod}\ N)\right\}
\end{equation}
for $N\ge3$ and $Y(N)=\Gamma(N)\backslash\H$ then $Y(N)$ admits the
structure of a smooth quasi-projective variety and there exists a
canonical family $E\to Y(N)$ of elliptic curves (with a level $N$
structure) over $Y(N)$ \cite[Chapter~7]{mf:git}. In this case the
projection $\pi:\H\to Y(N)$ is locally inverse to the ratio
$I_1(s)/I_2(s)$. Using Lemma~\ref{lem:clo-comp-inv} one can then show
that $\pi$ is Noetherian as well, when we identify $Y(N)$ with its
projective embedding. See the end of this section for a formal
treatment of this implication using the charts of~\secref{sec:charts}
in a more general setting. Of course, the Noetherianity of $\pi$ also
essentially follows from the construction of~\secref{sec:j-func}.
However, the present construction has the advantage of generalizing to
higher dimensions as we illustrate below.

Let $\H_g$ denote the Siegel half-space of genus $g$, consisting of
$g\times g$ symmetric matrices over $\C$ with positive-definite
imaginary part. The symplectic group $\Sp_{2g}(\R)$ acts on $\H_g$
(see e.g. \cite{bz:siegel}) and the quotient
$\cA_g:=\Sp_{2g}(\Z)\backslash\H_g$ is called the Siegel modular
variety. Then $\cA_g$ is the moduli space of principally polarized
abelian varieties of genus $g$, and it can be shown to have the
structure of a quasi-projective algebraic variety. We will show that
the uniformization of $\cA_g$ by $\H_g$ (given by the quotient map) is
a Noetherian map.

It will be convenient for our purposes to pass from $\Sp_{2g}(\Z)$ to
a finite-index subgroup such that the quotient becomes a smooth
manifold. Following \cite{bz:siegel} we take $\Gamma=\Gamma_{4,8}$ to
be the theta-group of level $(4,8)$,
\begin{multline}
  \Gamma_{4,8} := \Big\{
      \gamma=\begin{pmatrix}
        a&b\\c&d
      \end{pmatrix}\in\Sp_{2g}(\Z): \gamma={\boldsymbol1}_g(\text{mod}\ 4),\\
      \diag(a^tb)\equiv\diag(c^td)=0(\text{mod}\ 8)\Big\}.
\end{multline}
The quotient $S:=\Gamma\backslash\H_g$ is a complex manifold which has
a natural structure of a smooth quasi-projective variety
\cite[p.190]{mf:git}. The variety $S$ is equipped with a canonical
family $A\to S$ of principally-polarized abelian varieties with a
level $(4,8)$-structure \cite[Appendix~7,A--B]{mf:git}, and is the
moduli space for such varieties. We denote by $\pi:\H_g\to S$ the
canonical projection. Since $\Gamma$ is a finite index subgroup of
$\Sp_{2g}(\Z)$ there is a regular, finite algebraic map
$\kappa:S\to\cA_g$. By the results of~\secref{sec:charts} if we show
that $\pi$ is Noetherian then $\kappa\circ\pi$, which is the
uniformization $\H_g\to\cA_g$, is Noetherian as well. We proceed to
prove that $\pi$ is indeed a Noetherian map.

Fix $s_0\in S$ and let $\tau_0\in\H_g$ with $\pi(\tau_0)=s_0$.
According to \cite[Theorem~30.3]{shimura:book} we may choose $g$ rational
sections $\omega_1,\ldots,\omega_g:S\to\Hdr^1(A/S)$ which define a
basis of the holomorphic differentials on $A$ over a generic point of
$S$, are regular at $\tau_0$, and which admit
$2\pi i P(\tau)({\boldsymbol1}_g\ \tau)$ as a period
matrix\footnote{for an appropriate choice of basis for
  $H_1(A_{s_0},\C)$ with respect to the principal polarization} (where
we think of $\tau$ as a function of $s$) and $P(\tau_0)$ is an
invertible matrix. As periods, each of the entries of this matrix are
Noetherian functions in a neighborhood of $s_0$, and from
Lemma~\ref{lem:clo-inv} it follows that $\tau(s)$, being the product
of the second block by the inverse of the first, is also a Noetherian
function in a neighborhood of $s_0$.

Having shown that the local inverse of $\pi$ in a neighborhood of
$s_0$ is Noetherian, we may also conclude that $\pi$ is Noetherian in
a neighborhood of $\tau_0$. We do this in detail to illustrate the
technique involving Noetherian charts. Recall that
by~\secref{sec:charts} there exists a Noetherian chart $\Psi:U_0\to S$
with $U_0\subset\C^{\dim S}$, say mapping $0$ to $s_0$, such that
$\tau\circ\Psi:U_0\to\H_g$ is Noetherian. By construction
$\tau\circ\Psi$ is also invertible in a neighborhood of $0$, and by
Lemma~\ref{lem:clo-comp-inv} the inverse map $\tilde\pi:\H_g\to U_0$
mapping $\tau_0$ to $0$ is also Noetherian around $\tau_0$. But then
$\pi=\Psi\circ\tilde\pi$ is Noetherian in a neighborhood of $\tau_0$
by Lemma~\ref{lem:clo-comp}, as claimed.

We remark that for explicit computation of the Noetherian parameters
one would need an explicit description of the universal family and
associated Gauss-Manin connection, which may be a non-trivial problem.
An alternative approach would be to derive Noetherian systems for
higher dimensional thetanulls analogous to the Halphen
system~\eqref{eq:halphen} directly, using the Riemann theta functions
and their various identities. Such an approach is pursued for genus
two in \cite{ohyama} and for general genus in
\cite{zudilin:thetanulls}, where the thetanulls of any genus are shown
to satisfy a nonlinear system of differential equations which may
indeed be regarded as a Noetherian system. Unfortunately this system
admits singularities at the zeros of the thetanulls and therefore does
not quite establish their Noetherianity in the entire Siegel
half-space -- but it does seem to indicate that this direct approach
is feasible.

\section{Analytic theory of Weierstrass polydiscs}
\label{sec:weierstrass-analytic}

We begin by recalling the notion of a Weierstrass polydisc from
\cite{me:rest-wilkie}. We call a system $\vx$ of coordinates on $\C^n$
\emph{standard} if it is obtained from the standard coordinates by an
affine unitary transformation.

\begin{Def}\label{def:weierstrass}
  Let $X\subset\Omega$ be an analytic subset of pure dimension $m$. We
  say that a polydisc $\Delta=\Delta_z\times\Delta_w$ in a standard
  coordinate system $\vx=\vz\times\vw$ coordinates is a
  \emph{Weierstrass polydisc} for $X$ if $\dim\vz=m$,
  $\bar\Delta\subset\Omega$ and
  $(\bar\Delta_z\times\partial\Delta_w)\cap X=\emptyset$. We call
  $\Delta_z$ the base and $\Delta_w$ the fiber of $\Delta$.
\end{Def}

We will denote by $\pi_z:\C^n\to\C^m$ the projection map, and
by $\pi_z^X:=\pi_z\rest{X\cap\Delta}$.

\begin{Fact}[\protect{\cite[Fact~5]{me:rest-wilkie}}]\label{fact:weierstrass}
  Let $\Delta$ be a Weierstrass polydisc for $X$. Then
  $\pi_z^X:X\cap\Delta\to\C^m$ is a proper $e(X,\Delta)$-to-1 map for
  some number $e(X,\Delta)\in\N$ called the \emph{degree}.
\end{Fact}

In \cite{me:interpolation,me:rest-wilkie} Weierstrass polydiscs play a
key role in the study of rational points on analytic varieties.
Proving the existence of Weierstrass polydiscs with effective
estimates on the size and degree is the main step in establishing an
effective Pila-Wilkie result using this method. In this section we
develop some analytic tools to approach this problem. Namely, we
recall the notion of \emph{Bernstein index} of an analytic function,
and show that estimates on the Bernstein indices of appropriately
constructed functions imply the existence of Weierstrass polydiscs.
Later, in~\secref{sec:weierstrass-noetherian} we show that in the
Noetherian category one can indeed estimate the relevant Bernstein
indices in terms of the Noetherian parameters.

\subsection{Bernstein indices}

We recall the following definition from \cite{ny:london}.

\begin{Def}\label{def:bernstein}
  Let $U\subset\C$ be a domain with a connected piecewise smooth
  boundary and $K\subset U$ a compact subset. The \emph{Bernstein
    index} of a holomorphic function $f:\bar U\to\C$ with respect
  to the pair $K\subset U$ is the number
  \begin{equation}
    \fB_{K,U}(f) := \ln \frac{M_U(f)}{M_K(f)}, \qquad M_A(f):=\max_{z\in\bar A} |f(z)|.
  \end{equation}
  For $U$ a disc and $\eta>1$ called the \emph{gap}, we denote
  \begin{equation}
    \fB_U^\eta(f) := \fB_{\bar U^\eta,U}(f).
  \end{equation}
\end{Def}

The following Theorem~\ref{thm:bernstein-zeros},
Lemma~\ref{lem:bernstein-bd-iy} of \cite{iy:real-zeros} and
Lemma~\ref{lem:bernstein-subadd} of \cite{ny:london} hold, with
suitable constants, for arbitrary pairs $K\subset U$ as in
Definition~\ref{def:bernstein}. We will only require them in the case
of concentric discs and state them for this case in order to give
explicit asymptotics for the constants (the constants as given in
\cite{iy:real-zeros} are fully explicit and we use asymptotic notation
only to simplify the presentation). 

\begin{Thm}[\protect{\cite[Lemma~1,~Example~1]{iy:real-zeros}}]\label{thm:bernstein-zeros}
  Let $U$ be a disc and $\e>0$. For any $f\in\cO(\bar U)$ we have
  \begin{equation}
    \#\{z\in \bar U^{1+\e}:f(z)=0\} \le \gamma_\e\cdot \fB_U^{1+\e}(f), \qquad \gamma_\e = \frac2{\e^2}+O(\e).
  \end{equation}
  Here each root of $f$ is counted with its multiplicity.
\end{Thm}

The following lemma of \cite{iy:double-exp} gives a lower bound for
the values of a holomorphic function in terms of its Bernstein index.

\begin{Lem}[\protect{\cite[Lemma~3]{iy:real-zeros}}]\label{lem:bernstein-bd-iy}
  Let $U$ be a disc of radius $1$ and $\e>0$. For any
  $f\in\cO(\bar U)$ and any $h>0$ one can find a finite union $D_h$ of
  discs around roots of $f$ in $U$, with the sum of the diameters less
  than $h$, such that
  \begin{equation}
    \min_{z\in \bar U^{1+\e}\setminus D_h} |f(z)| \ge M \big(\frac m M\big)^{\chi_\e-\tau_\e\ln h}
  \end{equation}
  where $M=M_U(f)$, $m=M_{U^{1+\e}}(f)$ and
  \begin{align}
    \chi_\e &= \frac8{\e^4}\ln\frac1\e+O(\e^{-4}) & \tau_\e &= \frac2{\e^2}+O(\e).     
  \end{align}
\end{Lem}
\begin{proof}
  For the computation of the constants we note that (in the notation
  of \cite[Lemma~3]{iy:real-zeros}) we choose $V$ to be a disc of
  radius $1-\frac\e2$ and by \cite[Example~1]{iy:real-zeros} we have
  \begin{align}
    \rho &=1-\frac12\e^2+O(\e^3) & \gamma &= \frac2{\e^2}+O(\e) & \sigma &= \frac4{\e^2}+O(\e)
  \end{align}
  and then
  \begin{align}
    \lambda_1&=\frac2\e+O(1) & \lambda_2&=1-\frac12\e+O(\e^2) & \theta&=\frac8{\e^4}\ln\frac1\e+O(\e^{-4})
  \end{align}
  from which the claim follows.

  We also note that the fact that the centers of the discs can be
  chosen to be roots of $f$ is not part of the original statement of
  \cite[Lemma~3]{iy:real-zeros} but it is evident from the proof.
\end{proof}

As an easy consequence we have the following.

\begin{Lem}\label{lem:bernstein-bd-disc}
  Let $U$ be a disc and $f\in\cO(\bar U)$ with $M_U(f)=1$. Then there
  exists a disc $U^4\subset D\subset \bar U^2$ concentric with $U$
  such that
  \begin{equation}
    \min_{z\in\partial D} |f(z)| \ge e^{-O(\fB_U^2(f))}.
  \end{equation}
\end{Lem}
\begin{proof}
  Since the claim is invariant under rescaling we may suppose that $U$
  is a disc of radius $1$. Apply Lemma~\ref{lem:bernstein-bd-iy} with
  $h=1/4$ and note the $D_h$ must be disjoint from the radius of some
  disc concentric with $U$ with radius $1/4<r\le1/2$.
\end{proof}

We record a simple corollary of Lemma~\ref{lem:bernstein-bd-iy} which
allows one to control the Bernstein index with the standard gap $2$
in terms of the Bernstein index with a smaller gap.

\begin{Cor}\label{cor:bernstein-gap}
  Let $U$ be a disc and $f\in\cO(\bar U)$. Then for any $\e>0$ we have
  \begin{equation}
    \fB_U^2(f) \le (\chi_\e+\tau_\e\ln 2)\cdot\fB_U^{1+\e}(f).
  \end{equation}
\end{Cor}
\begin{proof}
  Since the claim is invariant by rescaling we may assume that $U$ is
  the unit disc. Apply Lemma~\ref{lem:bernstein-bd-iy} to obtain a
  finite union $D$ of discs with the sum of the diameters less than
  $1/2$, such that
  \begin{equation}
    \min_{z\in \bar U^{1+\e}\setminus D} |f(z)| \ge M \big(\frac m M\big)^{\chi_\e+\tau_\e\ln 2}
  \end{equation}
  where $M=M_U(f)$, $m=M_{U^{1+\e}}(f)$. In particular, since $D$
  cannot cover $U^2$, we have
  \begin{equation}
    \frac{M_{U^2}(f)}{M_U(f)} \ge \big(\frac m M\big)^{\chi_\e+\tau_\e\ln 2}.
  \end{equation}
  The claim now follows by taking inverse and log.
\end{proof}

We will also require the following subadditivity property of
\cite[Lemma~3]{ny:london}. We will not need to use the explicit form
of the constants, but we remark that they can be easily explicitly
recovered from the proof.

\begin{Lem}\label{lem:bernstein-subadd}
  Let $U$ be a disc and $f_1,\ldots,f_p\in\cO(\bar U)$. Then
  \begin{equation}
    \fB_U^2(f_1\cdots f_n) \le O(\ln(n+1))\sum_{j=1}^p \fB_U^2(f_j).
  \end{equation}  
\end{Lem}

For dimension greater than one we introduce the following
version of the Bernstein index.

\begin{Def}\label{def:bernstein-Cn}
  Let $B\subset\C^n$ be a Euclidean ball centered at $p$ and
  $F:\bar B\to\C$ be holomorphic. We define
  \begin{equation}
    \fB_{K,B}(F) := \max_{L\ni p} \fB_{K\cap L,B\cap L}(f)
  \end{equation}
  where $L$ ranges over the complex lines containing $p$. For $\eta>1$
  called the \emph{gap}, we denote
  \begin{equation}
    \fB_B^\eta(f) := \fB_{\bar B^\eta,B}(f).
  \end{equation}
\end{Def}

\subsection{Weierstrass polydiscs for holomorphic hypersurfaces}

Let $B\subset\C^m$ be a Euclidean ball and $R:\bar B\to\C$ a
holomorphic function. Our goal in this section is to construct a
Weierstrass polydisc $\Delta$ for the hypersurface $X_R:=\{R=0\}$.
More specifically, we would like to control the size of $\Delta$ and
the degree $e(X_R,\Delta)$ in terms of the Bernstein index
$\fB_B^2(R)$.

\begin{Prop}\label{prop:R-weierstrass}
  Let $R:\bar B\to\C$ be a holomorphic function and set
  $\fB:=\fB_B^2(R)$. Then there exists a Weierstrass polydisc
  $\Delta$ for $X_R:=B\cap\{R=0\}$ such that
  \begin{gather}
    B^\eta\subset\Delta\subset B\text{ where } \eta=e^{O(\fB)}, \label{eq:R-weierstrass1} \\
    e(X_R,\Delta) = O(\fB). \label{eq:R-weierstrass2}
  \end{gather}
  Here~\eqref{eq:R-weierstrass2} holds even if we consider $X_R$ with
  its cycle structure, i.e. count each component of $X_R$ with its
  associated multiplicity as a zero of $R$.
\end{Prop}
\begin{proof}
  Since the claim is invariant under rescaling of $R$ and $B$ we may
  assume that $B$ is the unit ball at the origin and that the maximum
  of $R$ on $B$ is achieved at some point $p\in B$ with $|R(p)|=1$.
  Let $L$ denote the complex line passing through the center of $B$
  and the point $p$ and write $U=B\cap L$. Let $\vz=z_1\times\vz'$ be a
  system of Euclidean coordinates on $B$ where the origin corresponds
  to the center of $B$ and $\vz'=0$ corresponds to $L$.

  According to Lemma~\ref{lem:bernstein-bd-disc} there exists a disc
  $U^4\subset D\subset \bar U^2$ such that
  \begin{equation}\label{eq:f-B-lb}
    \min_{z_1\in\partial D} |R(z_1,0)| \ge e^{-O(\fB)}.
  \end{equation}
  Since $R$ is assumed to have maximum norm $1$ on $B$, it follows from
  the Cauchy estimates that
  \begin{equation}\label{eq:f'-lb}
    \norm{\pd{f}{\vz'}(z_1,z')} = O(1), \qquad \forall z_1\in\partial D,\norm{z'}<1/2.
  \end{equation}
  Combining~\eqref{eq:f-B-lb} and~\eqref{eq:f'-lb} we see that we may
  choose a polydisc $\Delta_w$ of polyradius $e^{-O(\fB)}$ around the
  origin such that $R$ does not vanish on $\Delta_w\times\partial D$.
  Then $\Delta:=\Delta_{z'}\times D$ is a Weierstrass polydisc for
  $X_R$. To estimate $e(X_R,\Delta)$ it will suffice to count the
  number of zeros of $R$ in the fiber $\vz'=0$, i.e. the number of
  zeros of $R$ restricted to $D\subset L$. This follows directly from
  Theorem~\ref{thm:bernstein-zeros}.
\end{proof}

\subsection{Weierstrass polydisc for an intersection with a hypersurface}

Let $X\subset\Omega$ be an analytic subset of pure dimension $m$ and
$\Delta=\Delta_z\times\Delta_w$ a Weierstrass polydisc for $X$. Let $F:\Omega\to\C$ be
a holomorphic function, and set
\begin{align}
  X_F &:= X\cap\Delta\cap\{F=0\}, & Y_F &:= \pi_z(X_F).
\end{align}

\begin{Lem}\label{lem:X-F-weierstrass}
  Suppose $\Delta_z'=\Delta_{z'}\times\Delta_{w'}\subset\Delta_z$ is a
  Weierstrass polydisc for $Y_F$. Then
  $\Delta':=\Delta_z'\times\Delta_w$ is a Weierstrass polydisc for
  $X_F$.
\end{Lem}
\begin{proof}
  Recall that
  \begin{equation}
    \Delta_{z'}\times\partial(\Delta_{w'}\times\Delta_w)=
    \big[\Delta_{z'}\times\partial\Delta_{w'}\times\Delta_w\big]\cup
    \big[\Delta'_z\times\partial\Delta_w\big].
  \end{equation}
  It thus suffices to note that $X_F$ does not meet
  $\Delta'_z\times\partial\Delta_w$ since $X$ does not meet
  $\Delta_z\times\partial\Delta_w$; and $X_F$ does not meet
  $\Delta_{z'}\times\partial\Delta_{w'}\times\Delta_w$ since its
  $\pi_z$ projection $Y_F$ does not meet
  $\Delta_{z'}\times\partial\Delta_{w'}$.
\end{proof}

We define the \emph{analytic resultant} of $F$ with respect to
$X,\Delta$,
\begin{equation} \label{eq:analytic-resultant}
  \cR_F = \cR(X,\Delta,F):\Delta_z\to\C, \qquad \cR_F(z) = \prod_{w:(z,w)\in X\cap\Delta} F(z,w).
\end{equation}
By Fact~\ref{fact:weierstrass} we see that $\cR_F$ is a holomorphic
function, and by definition
\begin{equation}
  Y_F = \{z\in\Delta_z : \cR_F(z)=0\}.
\end{equation}
Here equality holds even if we consider $Y_F$ and $\{\cR_F=0\}$ with
their natural cycle structures.
\begin{Prop}\label{prop:X-F-weierstrass}
  Let $B\subset\Delta_z$ be a Euclidean ball with the same center as
  $\Delta_z$, and set $\fB:=\fB_B^2(\cR_F)$. Then there exists a
  Weierstrass polydisc $\Delta':=\Delta'_z\times\Delta_w$ for $X_F$
  such that
  \begin{gather}
    B^\eta\subset\Delta'_z\subset B\text{ where } \eta=e^{O(\fB)}, \label{eq:X-F-weierstrass1} \\
    e(X_F,\Delta') = O(\fB). \label{eq:X-F-weierstrass2}
  \end{gather}  
\end{Prop}
\begin{proof}
  By Proposition~\ref{prop:R-weierstrass} applied to $\cR_F$, there
  exists a Weierstrass polydisc $\Delta'_z\subset B$ for $Y_F$
  satisfying~\eqref{eq:X-F-weierstrass1}. Then by
  Lemma~\ref{lem:X-F-weierstrass}, $\Delta'$ is a Weierstrass polydisc
  for $X_F$. Finally $e(X_F,\Delta')$ is equal by definition to
  $e(Y_F,\Delta_z')$ if we count each component of $Y_F$ with its
  associated multiplicity, and~\eqref{eq:X-F-weierstrass2} follows
  from~\eqref{eq:R-weierstrass2} and the remark following it. 
\end{proof}

Proposition~\ref{prop:X-F-weierstrass} can be used to inductively
construct a Weierstrass polydisc of controllable size and degree for
the zero locus of a collection of functions, assuming one can
explicitly estimate the Bernstein indices $\fB$ involved.
Unfortunately it appears that the techniques presently at our disposal
do not suffice to produce such estimates for arbitrary collections of
Noetherian functions. Instead, we will focus on the case where $X$ is
an algebraic variety and $F$ is a Noetherian function, where a wider
range of techniques is available. As it turns out, this more
restrictive case will be sufficient for our purposes.

\section{Weierstrass polydiscs and Bernstein indices of Noetherian functions}
\label{sec:weierstrass-noetherian}

In this section we produce estimates for the Bernstein indices of
Noetherian functions and, more generally, their analytic
resultants~\eqref{eq:analytic-resultant} with respect to an algebraic
variety. In combination with the results
of~\secref{sec:weierstrass-analytic} this allows us to construct
Weierstrass polydiscs for the intersection between an algebraic
variety and a Noetherian hypersurface. The main statements are given
in~\secref{sec:wn-main}. Two principal ingredients from the
qualitative theory of differential equations are used in producing
these estimates. First, in~\secref{sec:wn-linear} we use some results
from the theory of linear differential equations to obtain
parametrizations of an algebraic curve which are well-behaved (in
terms of the differential equations involved). Consequently
in~\secref{sec:wn-curve-sys} we produce a well-behaved non-linear
differential equation for the restriction of a Noetherian function to
an algebraic curve. In~\secref{sec:nonlinear-ode} we introduce a
result of \cite{ny:chains} on the oscillation of trajectories of
(non-linear) polynomial vector fields. Finally
in~\secref{sec:wn-main-proof} we use this result to estimate the
Bernstein index of a Noetherian function restricted to an algebraic
curve, and consequently finish the proof of the main statement of
this section.

\subsection{Main statement}
\label{sec:wn-main}

Our goal in this section is to prove the following.

\begin{Thm}\label{thm:ntr-weierstrass}
  Let $V\subset\C^n$ be an algebraic variety of pure dimension $m$ and
  degree at most $\beta$. Let $F:\Omega\to\C$ be a Noetherian function
  of degree at most $\beta$. Let $B\subset\Omega$ be a Euclidean ball.
  Write
  \begin{equation}\label{eq:ntr-weierstrass-X}
    V_F:=(\Omega\cap V\cap\{F=0\})^{m-1}.
  \end{equation}
  Then there exists a Weierstrass polydisc $\Delta$ for $V_F$ and
  $\eta>0$ such that
  \begin{enumerate}
  \item $B^\eta\subset\Delta\subset B$ where $\eta=e^{C(\vphi,\beta)}$,
  \item $e(V_F,\Delta)=C(\vphi,\beta)$.
  \end{enumerate}
\end{Thm}

Theorem~\ref{thm:ntr-weierstrass} implies that one can cover
any compact piece of $V_F$ by an explicitly estimated number of
Weierstrass polydiscs with explicitly estimated degrees. The special
case $\dim V=1$ is of some independent interest and we record it
separately.

\begin{Cor}\label{cor:ntr-curve}
  Let $C\subset\C^n$ be an algebraic curve and $F:\Omega\to\C$ a
  Noetherian function, both of degree at most $\beta$. Let
  $B\subset\Omega$ be a Euclidean ball. Then for
  $\eta=e^{C(\vphi,\beta)}$ the number of isolated zeros of $F$ in the
  set $V\cap B^\eta$ is bounded by $C(\vphi,\beta)$.
\end{Cor}
\begin{proof}
  We observe that in this case $C_F$ of~\eqref{eq:ntr-weierstrass-X}
  has dimension zero and consists of the isolated zeros of $F$ on
  $C\cap\Omega$. In this context a Weierstrass polydisc $\Delta$ is
  just a disc with $\partial\Delta\cap C_F=\emptyset$ and
  $e(C_F,\Delta)$ is the number of points in $C_F\cap\Delta$. The
  claim follows since $B^\eta\subset\Delta$.
\end{proof}

By a covering argument, Corollary~\ref{cor:ntr-curve} allows one to
explicitly estimate the number of zeros of a Noetherian function
restricted to a compact piece of an algebraic curve. Crucially, the
estimate depends only on the degree of the curve and the Noetherian
parameters. We conjecture that a similar statement, with the curve $V$
replaced by the zero locus of arbitrary additional Noetherian
functions, is likely to be true. A conjecture in this spirit (in a
more local setting) is due to Gabrielov and Khovanskii \cite{GabKho},
with some partial results established in
\cite{GabKho,me:ntr-dim2,me:ntr-deformations}.

\subsection{Linear ODEs and algebraic functions}
\label{sec:wn-linear}

Let $y(t)$ be an algebraic function defined by the polynomial
$P(y,t)=0$ of degree $d$. It is classically known that $y$ satisfies a
scalar linear differential equation $L(y)=0$ of order $k\le d$,
\begin{equation}
  L = a_0(t)\partial_t^k+\cdots+a_k(t)y, \qquad a_0,\ldots,a_k\in\C[t], \quad a_0\not\equiv0.
\end{equation}
Following \cite{me:inf16} we define the \emph{slope}\footnote{we use
  $\ell_\infty$ rather than $\ell_2$-norms for convenience but this is
  of no significance} of $L$ to be
\begin{equation}
  \angle L := \max_{i=1,\ldots,k} \frac{\norm{a_i}_\infty}{\norm{a_0}_\infty}.
\end{equation}
It is not difficult to see that the degrees of $a_0,\ldots,a_k$ can be
estimated in terms of $d$. It is less trivial, but still true, that
the same is true for the slope $\angle L$. More explicitly, we have
the following.

\begin{Thm}\label{thm:alg-ODE}
  The operator $L$ can be chosen such that
  \begin{gather}
    \deg a_0,\ldots,\deg a_k = d^{O(1)}, \\
    \angle L = 2^{2^{\poly(d)}}. \label{eq:alg-slope-bound}
  \end{gather}
\end{Thm}
\begin{proof}
  Let $\cP$ denote the space of polynomials of degree $d$ in the
  variables $y,t$, which we identify with their tuple of coefficients.
  Consider $P(y,t)$ as a general polynomial of degree $d$ with
  indeterminate coefficients $\vp\in\cP$. According to
  \cite[Corollary~3.3]{me:q-functions} the vector function
  $\vy=(y,y^2,\ldots,y^d)$ satisfies an integrable regular system of
  the form
  \begin{equation}
    \d\vy = \Omega\cdot\vy
  \end{equation}
  where $\Omega$ is a rational matrix one-form in the variables
  $\vp,t$ with coefficients from $\Q(\vp,t)$. Moreover, the degree of
  the entries is bounded by $d^{O(1)}$ and the complexity (i.e.
  maximal height of any of the coefficients) is bounded by
  $2^{d^{O(1)}}$ (this is similar, but much easier, than
  \cite[Theorem~9]{me:inf16}).

  By a standard reduction from linear first-order systems to
  high-order scalar equations (e.g. \cite[Lemma~5]{me:polyfuchs}) one
  can then derive a family of linear operators
  \begin{equation}
    L_\vp = a_0(\vp,t)\partial_t^k+\cdots+a_k(\vp,t)y, \qquad a_0,\ldots,a_k\in\C[t], \quad a_0\not\equiv0.
  \end{equation}
  of order $k\le d$, degree $d^{O(1)}$ and complexity $2^{d^{O(1)}}$
  such that $L_\vp(y(\vp,t))=0$. By \cite[Principal
  Lemma~33]{me:inf16} there exists a proper algebraic subset
  $\Sigma\subset\cP$ such that for $\vp\not\in\Sigma$ the slope
  $\angle L_\vp$ is bounded by some explicit constant of the form
  $2^{2^{\poly(d)}}$.

  It remains to consider the case $\vp\in\Sigma$. Note that in this
  case it is possible that $a_0(\vp,t)\equiv0$ so that the expression
  defining the slope of $L_\vp$ is not well-defined. However, this
  problem is only apparent. Let $\gamma:(\C,0)\to\cP$ be a
  one-parametric family that meets $\Sigma$ only at $\gamma(0)=\vp$ and consider
  the family $L_s:=L_{\gamma(s)}$. Then $L_s\neq0$ for $s\neq0$, but
  $L_0$ may vanish identically. Let $\nu$ denote the order of vanishing
  in $s$, so that
  \begin{equation}
    L_s = s^\nu \tilde L_s, \qquad \tilde L_0\neq0.
  \end{equation}
  For $s\neq0$ we have
  \begin{equation}
    \tilde L_s(y(\gamma(s),t))=s^{-\nu}L_s(y(\gamma(s),t))=0
  \end{equation}
  and since both $\tilde L_s$ and $y(\gamma(s),t)$ are continuous
  (even analytic) in $s$ it follows that $\tilde L_0(y(\vp,t))=0$. We
  now note that since the slope is invariant under scalar
  multiplication, $\angle\tilde L_s=\angle L_s$ for $s\neq0$ and is
  therefore bounded by the uniform constant $2^{2^{\poly(d)}}$ as
  above. The slope of the limit $\tilde L_0$ is therefore bounded by
  the same constant, and $\tilde L_0$ thus satisfies the
  conditions of the lemma.
\end{proof}

In \cite{me:inf16} a result of the type of Theorem~\ref{thm:alg-ODE}
is proved for a much more general class of functions known as
\emph{Q-functions}, which includes the algebraic functions as well as
abelian integrals. The qualitative theory of linear ODEs is then used
to estimate the number of zeros of Q-functions, and similar ideas
could be used to give estimates for Bernstein indices as well. In this
approach the boundedness of the slope plays a key role. However, in
the context of the present paper we must consider \emph{Noetherian}
functions which satisfy \emph{non-linear} differential equations,
making the class of Q-functions inadequate for our purposes. As we
shall see, the boundedness of the slope for algebraic functions will
play a key role none the less.

\subsection{Bernstein indices for non-linear polynomial ODEs}
\label{sec:nonlinear-ode}

We consider a polynomial non-linear system of ODEs,
\begin{equation}\label{eq:nonlinear-ode}
  \partial_t x = \xi(t,x)
\end{equation}
where $\xi$ is a polynomial vector field on $\C_t\times\C^N$,
\begin{equation}
  \xi = \pd{}{t}+\xi_1 \pd{}{x_1}+\cdots+\xi_N \pd{}{x_N}, \qquad \xi_1,\ldots,\xi_N\in\C[t,x_1,\ldots,x_N].
\end{equation}
Let
\begin{gather}
  d = \deg\xi :=\max_{i=1,\ldots,N} \deg\xi_i, \\
  \norm\xi_\infty := \max_{i=1,\ldots,N} \norm{\xi_i}_\infty.
\end{gather}
We will require the following result of \cite{ny:chains}.

\begin{Thm}[\protect{\cite[Theorem~2]{ny:chains}}]\label{thm:nonlinear-bernstein}
  Let $S>2$ and $p=(t_0,x_0)\in\C_t\times\C^N$. Suppose that
  $\norm\xi_\infty\le S$ and $\norm{p}_\infty\le S$. Denote by
  $x=x(t)$ the solution of~\eqref{eq:nonlinear-ode} passing through
  $p$.

  Let $P\in\C[t,x_1,\ldots,x_N]$ have degree bounded by $d$ and
  suppose $P(t,x(t))\not\equiv0$. Then $x(t)$ can be extended to a
  disc $D=D_\rho(t_0)$ where
  \begin{gather}
    \rho = S^{-\exp^{\circ4}(O(Nd))}, \label{eq:chain-rho-bound}\\
    \fB_D^2(P(t,x(t))) \le d^{N^{O(N^2)}}.
  \end{gather}
  More generally, if $x(t)$ can be extended to a disc $D=D_\rho(t_0)$
  for some $\rho>0$ and remains bounded by $2S$ then
  \begin{equation}
    \fB_D^2(P(t,x(t))) \le d^{N^{O(N^2)}}+\rho\cdot S^{\exp^{\circ4}(O(Nd))}
  \end{equation}
\end{Thm}
\begin{proof}[Proof sketch]
  After increasing the dimension we may think of the coefficients of
  $P$ as extra variables in a polynomial ring $R$. Consider the chain
  of $R$-ideals defined recursively as follows
  \begin{equation}
    I_0=\<P\>, \qquad I_{k+1} = \<I_k,\xi^kP\>.
  \end{equation}
  Since $S$ is a noetherian ring, the chain $I_k$ must stabilize.
  Moreover, it follows from the Leibnitz the if $I_k=I_{k+1}$ then
  already $I_k=I_\infty$. Using methods of effective commutative algebra it
  is possible to give an effective upper bound for the first index
  of stabilization $k$ in terms of $N$ and $d$,
  \begin{equation}
    k = d^{N^{O(N^2)}}.
  \end{equation}
  With this $k$ stabilization implies that we have an equation
    \begin{equation}
    \xi^{k+1} P = \sum_{j=0}^k c_j\cdot \xi^{k-j} P, \qquad c_j\in R
  \end{equation}
  and restricting to the solution $(t,x(t))$ we have
  \begin{equation}\label{eq:chain-ODE}
    \partial_t^{k+1} P(t,x(t)) = \sum_{j=0}^k c_j(t,x(t))\cdot \partial_t^{k-j} P(t,x(t)).
  \end{equation}
  Moreover, by methods of effective commutative algebra it is possible
  to give an effective upper bound for the $\ell_\infty$-norms of the
  polynomials $c_j$ in terms of $N,d,S$. By simple growth estimates we
  can then choose $\rho$ satisfying~\eqref{eq:chain-rho-bound} such
  that $(t,x(t))$ remains in the ball of radius $2S$ for
  $t\in D=D_\rho(t_0)$. For any $\rho$ that satisfies this condition,
  the coefficients of~\eqref{eq:chain-ODE} are explicitly bounded by
  $S^{\exp^{\circ4}(O(Nd))}$ in $D$.
  
  We view~\eqref{eq:chain-ODE} as a linear differential operator $L$
  with holomorphic coefficients in $D$ satisfying $L(P(t,x(t)))=0$. We
  pass to the coordinate $s=(t-t_0)/\rho$ so that $D$ corresponds to
  the unit disc in $s$. We express $L$ in the $s$-coordinate and
  multiply by $\rho^{k+1}$ to obtain a monic operator $\tilde L$. Then
  the coefficients of $\tilde L$ are bounded by
  $\rho\cdot S^{\exp^{\circ4}(O(Nd))}$ in the unit disc, and the proof
  is then concluded by applying Theorem~\ref{thm:linear-bernstein}
  below. We remark that at the final step \cite{ny:chains} uses a
  lemma of Kim to estimate the number of zeros in $D$ and state the
  conclusion concerning this number rather than the Bernstein index,
  but this is of course a minor technical difference.
\end{proof}

To finish the proof of Theorem~\ref{thm:nonlinear-bernstein} we
recall the following theorem of \cite{ny:london}.

\begin{Thm}[\protect{\cite[Theorem~1]{ny:london}}]\label{thm:linear-bernstein}
  Let $K,U$ be as in Definition~\ref{def:bernstein} and let $L$ be a linear
  differential operator
  \begin{equation}
    L = \partial_t^{k+1}+\sum_{j=1}^k a_j(t)\partial_t^{n-j}, \qquad a_0,\ldots,a_k:\bar U\to\C
  \end{equation}
  with holomorphic coefficients of absolute value bounded by $M$ in
  $U$. Then for any solution $Lf=0$ we have
  \begin{equation}
    \fB_{K,U}(f) \le O(M+k\ln k)
  \end{equation}
  where the asymptotic constant depends only on $K,U$.
\end{Thm}

We remark that in \cite{ny:london}, Theorem~\ref{thm:linear-bernstein}
is stated as a bound on the number of zeros of $f$ rather than the
Bernstein index. However, the proof goes through an estimate for the
Bernstein index, which is given in the second displayed equation in
\cite[p.~317]{ny:london}.

\subsection{Restriction of a Noetherian function to an algebraic curve}
\label{sec:wn-curve-sys}

Let $C\subset\C^n$ be an algebraic curve of degree $d$, and let
$\pi_t:C\to\C_t$ be the restriction to $C$ of some affine coordinate
on $\C^n$ which is not constant on any component of $C$. Then there is
a (ramified) inverse map $t\to(x_1(t),\ldots,x_n(t))$ from $\C_t$ to
$C$ where each $x_j(t)$ is an algebraic function of degree at most
$d$. For $j=1,\ldots,n$ we choose a linear differential equation
$L^j(x_j(t))=0$ of order $k_j\le d$,
\begin{equation}
  L^j = a^j_0(t)\partial_t^{k_j}+\cdots+a^j_{k_j}(t)y, \qquad a^j_0,\ldots,a^j_{k_j}\in\C[t], \quad \norm{a^j_0}_\infty=1.
\end{equation}
where the operators $L_j$ satisfy the estimates of
Theorem~\ref{thm:alg-ODE}.

Recall that $\vphi:\Omega\to\C$ is a collections of Noetherian
functions satisfying~\eqref{eq:noetherian-sys}. We will study the
restriction $\phi\rest C$ by writing a non-linear system of the
form~\eqref{eq:nonlinear-ode} for the map $t\to(\vx(t),\vphi(x(t)))$.
Toward this end we write $N=k_1+\cdots+k_n+\ell$ and work in the
ambient space $\C^N$ with the coordinates $x_j\^k$ and $Q_j$ for
$j=1,\ldots,n$ and $k=1,\ldots,k_j-1$ and $\phi_1,\ldots,\phi_\ell$. In this
space we introduce the following system,
\begin{equation}\label{eq:phi-C-system}
  \begin{aligned}
    \partial_t x_j\^k &= x_j\^{k+1} \qquad k=1\ldots,k_j-2 \\
    \partial_t x_j\^{k_j-1} &= -Q_j(a_1^{k_j-1}x_j\^{k_j-1}+\cdots+a_{k_j}x_j\^0) \\
    \partial_t Q_j &= -\partial_t(a_0^j)\cdot (Q_j)^2 \\
    \partial_t \phi_l &= P_{l,1}(\vx\^0,\vphi)\cdot x_1\^1+\cdots+P_{l,n}(\vx\^0,\vphi)\cdot x_n\^1.
  \end{aligned}
\end{equation}
Let $\xi_{\vphi,C}$ denote the vector field corresponding
to~\eqref{eq:phi-C-system} in the sense of \secref{sec:nonlinear-ode}.
Then it is straightforward to verify using the estimates of
Theorem~\ref{thm:alg-ODE} that
\begin{equation}\label{eq:phi-C-sys-bound}
  \norm{\xi_{\vphi,C}}_\infty = 2^{2^{\poly(d)}}\cdot\NS(\vphi).
\end{equation}

Let $t\in\C_t$ be a point where $\pi_t\rest C$ is unramified and the
polynomials $a_0^1,\ldots,a_0^n$ are non-vanishing, and let
$\tilde\vx(t)$ be an analytic branch of the algebraic map $\vx(t)$.
Then it is straightforward to check that the map $\Phi:\C_t\to\C^N$
defined by
\begin{align}\label{eq:Phi-def}
  x_j\^k &= \partial_t^k \tilde x_j(t) & Q_j &= 1/a_0^j(t) & \phi_l &= \phi_l(\tilde\vx(t))
\end{align}
is a solution of~\eqref{eq:phi-C-system}.

\subsection{Proof of Theorem~\ref{thm:ntr-weierstrass}}
\label{sec:wn-main-proof}

Let $V\subset\C^n$ be an algebraic variety of pure dimension $m$ and
degree $d$. Let $F:\Omega\to\C$ be a Noetherian function of degree
$d$.. Finally let $B\subset\Omega$ be a Euclidean ball. We recall the
following result of \cite{me:rest-wilkie}.

\begin{Thm}[\protect{\cite[Theorem~7]{me:rest-wilkie}}]\label{thm:algebraic-weierstrass}
  Let $B\subset\C^n$ be a Euclidean ball and $V\subset\C^n$ be an
  algebraic variety of pure dimension $m$ and degree $d$. Then there
  exists a Weierstrass polydisc $\Delta:=\Delta_z\times\Delta_w$ for
  $V$ with the same center as $B$ such that
  $B^{\tilde\eta}\subset\Delta\subset B$ and $\tilde\eta = d^{O(1)}$.
\end{Thm}

Note that Theorem~\ref{thm:algebraic-weierstrass} is originally stated
for general sub-Pfaffian sets, and above we give only the algebraic
case which will suffice for our purposes. We also note that the
theorem is originally stated for a ball around the origin, but this is
clearly of no significance in the formulation above. Finally, in our
formulation we implicitly used the fact the algebraic varieties of
degree $d$ are set-theoretically cut out by polynomials of degree at
most $d$, see e.g. \cite[Lemma~29]{me:rest-wilkie}.

We remark that Theorem~\ref{thm:algebraic-weierstrass} could also be
established inductively by the methods used in this paper, but the
estimates obtained in this way would be significantly weaker.

Let $\Delta\subset B$ be a Weierstrass polydisc for $V$ as in
Theorem~\ref{thm:algebraic-weierstrass}. Then $\Delta\cap V$
decomposes into a union of irreducible analytic components. We denote
by $\tilde V\subset\Delta$ the union of these components where $F$
does not vanish identically, so that
\begin{equation}\label{eq:V_F-tilde-V}
  V_F\cap\Delta = \tilde V \cap \{F=0\}.
\end{equation}
By definition $\Delta$ is also a Weierstrass polydisc for $\tilde V$.
We let $\cR_F:=\cR(\tilde V,\Delta,F)$ denote the analytic
resultant~\eqref{eq:analytic-resultant} of $F$ with respect to
$\tilde V,\Delta$. Denote by $B_z\subset\Delta_z$ the largest ball in
$\Delta_z$ with the same center. Evidently
\begin{equation}\label{eq:Deltaz-vs-Bz}
  \Delta_z^{O(1)} \subset B_z.
\end{equation}
We will study the restriction of $\cR_F$ to $B_z$.

Let $L\in\C^m$ be a complex line through the center of $B_z$ and
$C\subset\C^n$ be the complex curve $C=V\cap\pi_z^{-1}(L)$. Let
$\pi_t:\C^n\to\C$ be an affine combination of the $\vz$ coordinates
which maps $L\cap B_z$ onto the unit disc $D$. Denote by $\Sigma$ the
ramification locus of $\pi_t\rest C$. Consider the
system~\eqref{eq:phi-C-system} for the pair $C,\pi_t$. Then for any
$t_0\in D\setminus\Sigma$ there exist exactly $\nu\le d$ points such
that
\begin{equation}
  p_1(t_0),\ldots,p_\nu(t_0) \in \tilde V\cap\pi_t^{-1}(t_0).
\end{equation}
Moreover, these points extend as ramified algebraic functions for
$t\not\in\Sigma$.

Suppose that $t_0$ is also not a root of the polynomials
$a_0^1,\ldots,a_0^n$. Then for $i=1,\ldots,\nu$, we have a map
$\Phi_i:\C_t\to\C^N$ defined as in~\eqref{eq:Phi-def} with
$\tilde x(t)=p_i(t)$. In general, $\norm{\Phi_i(t_0)}_\infty$ cannot
be bounded in terms of the Noetherian parameters alone: it tends to
infinity as $t_0$ tends to $\Sigma$ or to a zero of a polynomial
$a_0^j$. However, we will show that on a suitably chosen annulus
around the origin one can indeed control the norms. The key estimate
is contained in the following lemma.

\begin{Lem}\label{lem:a_0^j-bounds}
  One can choose $1/2<r<3/4$ and $\rho=d^{-O(1)}$ such that
  $A_{r,\rho}:=\{r-\rho<|t|<r+\rho\}$ satisfies
  \begin{align}
    \dist(A_r,\Sigma)&=d^{-O(1)} &  \min_{\substack{t\in A_{r,\rho}\\j=1,\ldots,n}} |a_0^j(t)| &= e^{-d^{O(1)}}.
  \end{align}
\end{Lem}
\begin{proof}
  Recall that we have $\norm{a^j_0}_\infty=1$ and
  $\deg a_0^j=d^{O(1)}$. Then
  \begin{align}
    M_D(a_0^j)&\ge 1 & M_{D^{1/2}}(a_0^j)&\le 2^{d^{O(1)}},
  \end{align}
  where the lower bound follows from the Cauchy formula on the unit
  disc and the upper bound is straightforward. Therefore
  $\fB_D^2(a_0^j)\le d^{-O(1)}$. Then by
  Lemma~\ref{lem:bernstein-bd-iy} we can choose a union $D_j$ of discs
  with the sum of the diameters less than $1/(9n)$ such that
  \begin{equation}
    \min_{t\in D\setminus D_j}|a_0^j(t)|\ge e^{-d^{O(1)}}.
  \end{equation}
  Moreover the center of each disc is a root of $a^j_0$ and in
  particular the number of discs does not exceed $d^{O(1)}$.

  We also have $\#\Sigma=d^{O(1)}$, for instance since every point of
  $\Sigma$ is a root of some $a_0^j$. We let $D'$ denote the union of
  discs of radius $1/(9\cdot\#\Sigma)$ around each point of $\Sigma$,
  so that the sum of diameters is at most $1/9$ and we have
  \begin{equation}
    \dist(\Sigma,D\setminus D') = d^{-O(1)}.
  \end{equation}

  In the collection $D',D_1,\ldots,D_n$ we have $N=d^{O(1)}$ discs
  with the sum of the diameters at most $2/9$. It is then a simple
  geometric exercise to show that one can choose an annulus
  $A_{r,\rho}$ with $1/4<r<1/2$ and $\rho=d^{-O(1)}$ which is disjoint
  from the union, which concludes the proof.
\end{proof}

We are now ready to establish an upper bound for
$\norm{\Phi_i(t)}_\infty$. We choose and fix $A:=A_{r,\rho}$ as in
Lemma~\ref{lem:a_0^j-bounds}.

\begin{Lem}\label{lem:Phi_i-bound}
  For every $i=1,\ldots,\nu$ and $t_0\in A$ we have
  $\norm{\Phi_i(t_0)}_\infty=e^{d^{O(1)}}\cdot\NS(\vphi)$.
\end{Lem}
\begin{proof}
  The $x_j=x_j\^0$ and $\phi_l$ coordinates are bounded by
  $\NS(\vphi)$ by definition. The $Q_j$ coordinates are bounded by
  $e^{d^{O(1)}}$ by Lemma~\ref{lem:a_0^j-bounds}. The $x_j\^k$
  coordinates (for $0<k<k_j$) are given by
  $\partial_t^k (p_i)_j(t_0)$, where $(p_i)_j$ denotes the $x_j$
  coordinate of the branch $p_i$. Recall that $p_i$ extends
  holomorphically as long as $t\not\in\Sigma$, i.e. by
  Lemma~\ref{lem:a_0^j-bounds} to a disc of radius $d^{-O(1)}$.
  Moreover in this disc the image of $p_i$ remains in
  $C\cap\Delta\subset\Omega$ and in particular $|(p_i)_j(t)|$ is bounded by
  $\NS(\vphi)$ in this disc. Applying the Cauchy estimate for
  $\partial_t^k (p_i)_j(t_0)$ in the disc we obtain the bound
  $\NS(\vphi)\cdot d^{-O(k)}$ for $x_j\^k$, and since we have
  $k<k_j\le d$ this bound is of the required form.
\end{proof}

Let $|t_0|=r$ and $i=1,\ldots,\nu$. By the choice of $A$ the
trajectory $\Phi_i(t)$ can be extended to a disc $D_\rho(t_0)$ and by
Lemma~\ref{lem:Phi_i-bound} we have
\begin{equation}\label{eq:Phi_i-bound}
  \norm{\Phi_i(t)}_\infty = e^{-d^{O(1)}}\cdot\NS(\vphi), \qquad \forall t\in D_\rho(t_0).
\end{equation}
We recall also that by~\eqref{eq:phi-C-sys-bound} we have
\begin{equation}\label{eq:xi-vphi-C-bound}
  \norm{\xi_{\vphi,C}}_\infty =2^{2^{\poly(d)}}\cdot\NS(\vphi).
\end{equation}

We now return to the analysis of the analytic resultant $\cR_F$, and
more specifically its restriction to the complex line $L$. Working in
the $t$ coordinate for $t\in D$, we have by definition
\begin{equation}\label{eq:R_F-t}
  \cR_F(t) = F(p_1(t))\cdots F(p_\nu(t)).
\end{equation}
Since $F$ is a Noetherian function of degree $d$, the function
$F(p_i(t))$ can be written in the form $P(\Phi_i(t))$ for $P$ a
polynomial of degree $d$ on $\C^N$. Then by
Theorem~\ref{thm:nonlinear-bernstein},
using~\eqref{eq:Phi_i-bound},~\eqref{eq:xi-vphi-C-bound} and
$N\le nd+\ell=O(d)$, we have
\begin{equation}\label{eq:R-F-bern0}
  \fB^2_{D_\rho(t_0)}(P(\Phi_i(t)))= C(\vphi,d).
\end{equation}
Since~\eqref{eq:R-F-bern0} is true for $i=1,\ldots,\nu$ we have by
Lemma~\ref{lem:bernstein-subadd} also
\begin{equation}\label{eq:R_F-bern1}
  \fB^2_{D_\rho(t_0)}(\cR_F(t)) =  C(\vphi,d).
\end{equation}
Recall that $\cR_F$ is in fact a holomorphic function in $D$. In
particular, its maximum on the disc $\bar D_{r+\rho}(0)$ is obtained
somewhere on the boundary. For a suitable choice of $t_0$, this same
maximum is obtained on $D_\rho(t_0)$. On the other hand, the maximum
of $\cR_F$ on the disc $\bar D_{r+\rho/2}(0)$ is certainly no smaller
than the maximum on the disc $D_{\rho/2}(t_0)$. Thus
from~\eqref{eq:R_F-bern1} we see that
\begin{equation}\label{eq:R_F-bern2}
  \fB_{\bar D_{r+\rho/2}(0),D_{r+\rho}(0)}(\cR_F(t)) =  C(\vphi,d).
\end{equation}
The gap in the Bernstein index above is $\sim\rho$, and by
Corollary~\ref{cor:bernstein-gap} we have
\begin{equation}
  \fB_{D_{r+\rho}(0)}^2(\cR_F(t))=  C(\vphi,d).
\end{equation}
Since $r>1/2$ and $(r+\rho)/2<3/8+o(1)$ the middle index below has gap
uniformly bounded from zero, and we have again by
Corollary~\ref{cor:bernstein-gap},
\begin{equation}\label{eq:R_F-bern3}
  \fB^2_{D^2}(\cR_F(t)) = O(\fB_{\bar D_{(r+\rho)/2}(0),D_{1/2}(0)}(\cR_F(t)))=  C(\vphi,d).
\end{equation}
Since~\eqref{eq:R_F-bern3} holds for any complex line $L$ through the
center of $B_z$, and since $D^2$ corresponds in the $t$-chart to
$B_z^2\cap L$, we finally have
\begin{equation}\label{eq:R_F-bern4}
  \fB^2_{B_z^2}(\cR_F) =  C(\vphi,d).
\end{equation}
By~\eqref{eq:V_F-tilde-V} and Proposition~\ref{prop:X-F-weierstrass}
applied to the ball $B^2_z$, there exists a Weierstrass polydisc
$\Delta':=\Delta'_z\times\Delta_w$ for $V_F$ such that
\begin{gather}
  B_z^{2\eta'}\subset\Delta'_z\subset B^2_z\text{ where } \eta'=e^{C(\vphi,d)}, \label{eq:B_z-eta'}\\
  e(V_F,\Delta') = C(\vphi,d).
\end{gather}  
Finally, we deduce from~\eqref{eq:Deltaz-vs-Bz},~\eqref{eq:B_z-eta'}
and Theorem~\ref{thm:algebraic-weierstrass} that
\begin{equation}
  B^\eta\subset\Delta'\subset B, \qquad \eta = e^{C(\vphi,d)}.
\end{equation}
which concludes the proof.

\section{Rational and algebraic points on Noetherian varieties}
\label{sec:rational-points}

In this section we study rational (and more generally algebraic)
points on Noetherian varieties and prove Theorems~\ref{thm:main}
and~\ref{thm:main-k}.

\subsection{Rational points in a Weierstrass polydisc}

We begin by recalling the relation, established in
\cite{me:interpolation,me:rest-wilkie}, between Weierstrass polydiscs
and the study of rational points on analytic sets. Let
$X\subset\Omega$ be an analytic set of pure dimension $m$. Let
$\Delta:=\Delta_z\times\Delta_w$ be a Weierstrass polydisc for $X$ and
set $\Delta':=\Delta_z\times\Delta_w^{1/3}$ and $\nu:=e(X,\Delta)$.

\begin{Prop}\label{prop:hypersurface-select}
  Let $M,H\ge3$ and suppose $\vf:=(f_1,\ldots,f_{m+1})\in\cO(\bar\Delta')$
  satisfy $M_{\Delta'}(f_i)\le M$. Let
  \begin{equation}
    Y:=\vf(X\cap\Delta^{H^\e}) \subset \C^{m+1}.
  \end{equation}
  For every $\e>0$ there exists a number
  \begin{equation}\label{eq:hypersurface-select}
    d = O(\nu^{n-m}\e^{-m}(\log M)^m)
  \end{equation}
  such that $Y(\Q,H)$ is contained in an algebraic hypersurface of
  degree at most $d$ in $\C^{m+1}$.
\end{Prop}
\begin{proof}
  According to \cite[Proposition~11]{me:interpolation}, and plugging
  in the values of $\norm\cD,e(\cD)$ from
  \cite[Theorem~3]{me:rest-wilkie}, we see that it suffices
  to choose $d$ such that
  \begin{equation}
    \e\log H > C_1 \frac{d^{-1}\log(\nu^{n-m})+\log M+\log H}{(d/\nu^{n-m})^{1/m}}.
  \end{equation}
  In particular it is enough to have
  \begin{equation}
    d>(n-m)\log\nu \quad\text{and}\quad \e > C_1 \frac{\log M+2}{(d/\nu^{n-m})^{1/m}},
  \end{equation}
  which is compatible with~\eqref{eq:hypersurface-select}.
\end{proof}

\subsection{Exploring rational points in complex Noetherian varieties}

The following is our main result for this section.

\begin{Thm}\label{thm:main-complex}
  Let $X\subset\Omega$ be a Noetherian variety of degree $\beta$ and
  $\e>0$. There exist constants
  \begin{equation}
    d,N = C_n(\vphi,\beta\e^{1-n})
  \end{equation}  
  with the following property. For every $H\in\N$ there exist
  at most $NH^\e$ irreducible algebraic varieties $V_\alpha\subset\C^n$
  with $\deg V_\alpha\le d$ such that
  \begin{equation}
    X(\Q,H) \subset \bigcup_\alpha X(V_\alpha).
  \end{equation}  
\end{Thm}

The following proposition provides the key inductive step in the proof
of Theorem~\ref{thm:main-complex}.

\begin{Prop}\label{prop:Valpha-step}
  Let $W\subset\C^n$ be an irreducible algebraic variety of dimension
  $m+1$ of degree at most $\beta$ and let $X\subset\Omega\cap W$ be a
  Noetherian variety of degree at most $\beta$. Let $\e>0$. There
  exist constants
  \begin{align}
    d &= C(\vphi,\beta)\e^{-m} \\
    N &= e^{C(\vphi,\beta)}
  \end{align}
  with the following property. For every $H\in\N$ there exist
  at most $NH^\e$ hypersurfaces $\cH_\alpha\subset\C^n$ 
  with $\deg \cH_\alpha\le d$ such that $W\not\subset \cH_\alpha$ and
  \begin{equation}
    X(\Q,H) \subset X(W)\cup\bigcup_\alpha \cH_\alpha.
  \end{equation}  
\end{Prop}
\begin{proof}
  Let $\{F_i\}$ denote the finite collection of Noetherian functions
  of degrees bounded by $\beta$ such that $X$ is their common zero
  locus. As an analytic set, $W$ may contain several irreducible
  components $\{W_h\}$. We let $F$ denote a generic linear combination
  of the $F_i$ such that for every $h$, $F$ vanishes identically on
  $W_h$ if and only if every $F_i$ does. Set
  \begin{equation}
    W_F := (W\cap\{F=0\})^m.
  \end{equation}

  By \cite[Lemma~29]{me:rest-wilkie} there exists a hypersurface
  $\cH_0\subset\C^n$ containing $\Sing W$ and not $W$ with
  $\deg\cH_0\le\beta$. If $p\in X\setminus\cH_0$ then $W$ is smooth at
  $p$ and in particular the germ $W_p$ consists of a single analytic
  component. If $F$ vanishes identically on this component then by
  construction $W_p\subset X$, so that $p\in X(W)$. Otherwise
  $p\in W_F$, and it remains to construct a collection of
  hypersurfaces $\cH_\alpha$ as in the statement with
  \begin{equation}
    W_F(\Q,H) \subset \bigcup_\alpha \cH_\alpha.
  \end{equation}
  Set $S:=\NS(\vphi)$. Recall from Lemma~\ref{lem:boundary} that our
  Noetherian system extends to a $\rho$-neighborhood of $\Omega$ with
  $\rho=O(S^{-O(1)})$, and the Noetherian size of our system in this
  larger domain is at most $O(S)$. Let $p\in\Omega$ and let $B$ denote
  the ball of radius $\rho$ around $p$. According to
  Theorem~\ref{thm:ntr-weierstrass} there exists a Weierstrass
  polydisc $\Delta$ for $W_F$ and $\eta>0$ such that
  \begin{enumerate}
  \item $B^\eta\subset\Delta\subset B$ where $\eta=e^{C(\vphi,\beta)}$,
  \item $e(W_F,\Delta)=C(\vphi,\beta)$.
  \end{enumerate}
  We choose $m+1$ coordinates $\vf:=(f_1,\ldots,f_{m+1})$ among the
  standard coordinates on $\C^n$ such that the projection
  $\vf:W\to\C^{m+1}$ is dominant. We apply
  Proposition~\ref{prop:hypersurface-select} to $W_F,\Delta,\vf$ and
  $\e/n$ and note that $M=O(S)$ to conclude that
  \begin{equation}
    (W_F\cap B^{H^{\e/n}\eta})(\Q,H)\subset (\vf(W_F\cap\Delta^{H^{\e/n}}))(\Q,H) 
  \end{equation}
  is contained in an algebraic hypersurface of degree $d$ as in the
  statement, which does not contain $W$ since $\vf$ is dominant.

  Finally it remains to cover $\Omega$ by balls of the form
  $B^{H^{\e/n}\eta}$, i.e. balls of radius
  $H^{\e/n}\cdot\eta\cdot\rho=H^{\e/n}\cdot e^{C(\vphi,\beta)}$ with
  centers $p\in\Omega$, and take the collection of corresponding
  hypersurfaces. The domain $\Omega$ is contained in a ball of radius
  $S$, and a simple subdivision argument shows that this can be done
  with $NH^\e$ balls as above.
\end{proof}

The following lemma gives an inductive proof of
Theorem~\ref{thm:main-complex}, which is obtained for the case
$W=\C^n$.

\begin{Lem}\label{lem:Valpha-induction}
  Let $W\subset\C^n$ be an irreducible algebraic variety of dimension
  $m+1$ and degree at most $\beta$ and let $X\subset\Omega\cap W$ be a
  Noetherian variety of degree at most $\beta$. There exist constants
  \begin{equation}
    d,N = C_{m+1}(\vphi,\beta\e^{-m})
  \end{equation}
  with the following property. For every $H\in\N$ there exist
  at most $NH^\e$ irreducible algebraic varieties $V_\alpha\subset W$
  with $\deg V_\alpha\le d$ such that
  \begin{equation}
    X(\Q,H) \subset \bigcup_\alpha X(V_\alpha).
  \end{equation}  
\end{Lem}
\begin{proof}
  We proceed by induction on $m$, where the case $m=-1$ is trivial. By
  Proposition~\ref{prop:Valpha-step} applied to $X,W$ we have a
  collection of at most $e^{C(\vphi,\beta)}H^{\e/2}$ hypersurfaces
  $\cH_{\alpha'}$ of degrees $C(\vphi,\beta)\e^{-m}$ such that
  \begin{equation}
    X(\Q,H) \subset X(W)\cup\bigcup_{\alpha'} \cH_{\alpha'}.
  \end{equation}
  Let $\{W_\alpha\}$ denote the union over $\alpha'$ of the sets of
  irreducible components of $W\cap\cH_{\alpha'}$. The degree of the
  intersection is bounded by the product of the degrees, and since the
  number of irreducible components of a variety is bounded by its
  degree we have
  \begin{align}
    \deg W_\alpha &= C(\vphi,\beta)\e^{-m} & \dim W_\alpha&=m & \#\{W_\alpha\} = e^{C(\vphi,\beta)}\e^{-m} H^{\e/2}
  \end{align}
  and
  \begin{equation}
    X(\Q,H) \subset X(W)\cup\bigcup_\alpha (X\cap W_\alpha).
  \end{equation}
  We now apply the inductive hypothesis to each pair
  $W_\alpha,X\cap W_\alpha$ with the exponent $\e/2$ to obtain
  collections $W_{\alpha,\beta}$ with
  \begin{align}
    \deg W_{\alpha,\beta} &= C_m(\vphi,C(\vphi,\beta)\e^{-2m}) \\
    \#\{W_{\alpha,\beta}\} &=  C_m(\vphi,C(\vphi,\beta)\e^{-2m})\cdot H^{\e/2}
  \end{align}
  such that
  \begin{equation}
    (X\cap W_\alpha)(\Q,H) \subset X(W_{\alpha,\beta}).
  \end{equation}
  Finally we take $\{V_\alpha\}$ to be the union of the sets $\{W\}$
  and $\{W_{\alpha,\beta}\}$. 
\end{proof}

\subsection{Exploring algebraic points in complex Noetherian
  varieties}

Our goal in the section is to establish the following generalization
of Theorem~\ref{thm:main-complex}.

\begin{Thm}\label{thm:main-k-complex}
  Let $X\subset\Omega$ be a Noetherian variety of degree $\beta$ and
  $\e>0$. There exist constants
  \begin{equation}
    d,N = C_{n(k+1)}(\vphi,\beta\e^{1-n})
  \end{equation}  
  with the following property. For every $H\in\N$ there exist at most
  $NH^\e$ irreducible algebraic varieties $V_\alpha\subset\C^n$ with
  $\deg V_\alpha\le d$ such that
  \begin{equation}
    X(k,H) \subset \bigcup_\alpha X(V_\alpha).
  \end{equation}  
\end{Thm}

The proof of Theorem~\ref{thm:main-k-complex}, adapted from
\cite{pila:algebraic-points}, is given in the remainder of this
section. We begin by setting up some notation. Let
$\cP_{\le k}:=\R^{k+1}\setminus\{{\mathbf0}\}$. For
$\vc\in\cP_{\le k}$ let $P_\vc\in\R[x]$ denote the polynomial
\begin{equation}
  P_\vc(X) := \sum_{j=0}^k c_j X^j.
\end{equation}
Following \cite{pila:algebraic-points} we introduce the following
height function. For an algebraic number $\alpha\in\Qa$ we define
\begin{equation}\label{eq:Hpoly-def}
  H_k^\poly(\alpha) = \min\{H(\vc) : \vc\in\cP_{\le k}(\Q), \quad P_\vc(\alpha)=0\}
\end{equation}
and $H^\poly_k(\alpha)=\infty$ if $[\Q(\alpha):\Q]>k$. Then whenever
$[\Q(\alpha):\Q]\le k$ we have \cite[5.1]{pila:algebraic-points}
\begin{equation}\label{eq:H-vs-Hpoly}
  H_k^\poly(\alpha) \le 2^k H(\alpha)^k.
\end{equation}
We define $X^\poly(k,H)$ in analogy with $X(k,H)$ replacing $H(\cdot)$
by $H^\poly(\cdot)$. In light of~\eqref{eq:H-vs-Hpoly}, it will
suffice to prove the claim for $X^\poly(k,H)$.

Let $\Sigma\subset\C^n\times\cP_{\le k}^n$ by the algebraic variety given by
\begin{equation}
  \Sigma := \{ (\vx,\vc_1,\ldots,\vc_n) : P_{\vc_1}(x_1)=\dots=P_{\vc_n}(x_n)=0\}.
\end{equation}
and denote by $\pi_1,\pi_2$ the projections to
$\C^n,\cP_{\le k}^n$ respectively. Let
$Y:=\pi_1^{-1}(X)\cap\Sigma$. Note that $Y$ is a Noetherian variety of
degree $O(\beta)$ and Noetherian size $O(\NS(\vphi))$.

Set $\tilde\Omega = \Omega\times U^n_{\le k}$ where
$U_{\le k}\subset\cP_{\le k}$ is given by
\begin{equation}
  U_{\le k} = \{\vc : 1/2<\max_{j=0,\ldots,k}|c_j|<2\}.
\end{equation}
Denote $\tilde Y:=\tilde\Omega\cap Y$ and
\begin{equation}
   \tilde Y(\Q,H;\pi_2) := \{\vy\in\tilde Y:H(\pi_2(\vy))\le H\}.
\end{equation}
We claim that
\begin{equation}\label{eq:Xpoly-vs-Ytilde}
  X^\poly(k,H)\subset\pi_1[\tilde Y(\Q,H^2;\pi_2)].
\end{equation}
Indeed, let $\vx\in X^\poly(k,H)$, and for every coordinate $x_i$
choose the corresponding polynomials $P_{\vc_i}$ as
in~\eqref{eq:Hpoly-def}. Then the coefficient of each $\vc_i$ are
bounded by $H$, and we let $\vc_i'$ be the vector obtained by
normalizing them to have maximum $1$ so that $\vc_i\in U_{\le k}$.
Then we have $H(\vc_i')\le H^2$ so
$(\vx,\vc_1',\ldots,\vc_n')\in\tilde Y(\Q,H^2;\pi_2)$. We now turn to
the description of $\tilde Y(\Q,H;\pi_2)$.

\begin{Lem}\label{lem:Ytilde-blocks}
  There exist constants
  \begin{equation}
    d,N = C_{n(k+1)}(\vphi,\beta\e^{-m})
  \end{equation}  
  with the following property. For every $H\in\N$ there exist at most
  $NH^\e$ irreducible algebraic varieties
  $\tilde V_\alpha\subset\C^n\times\cP^n_{\le k}$ with
  $\deg \tilde V_\alpha\le d$ such that
  \begin{equation}
    \tilde Y(\Q,H;\pi_2) \subset \bigcup_\alpha \tilde Y(\tilde V_\alpha).
  \end{equation}
\end{Lem}
\begin{proof}
  The claim follows essentially by repetition of the proof of
  Theorem~\ref{thm:main-complex} with the following modification.
  Since $Y$ is a subset of the algebraic variety $\Sigma$ we begin our
  induction in Lemma~\ref{lem:Valpha-induction} with $W=\Sigma$
  rather than $W=\C^n\times\cP^n_{\le k}$. Note that
  $\dim\Sigma=\dim\cP^n_{\le k}=n(k+1)$.

  In the notations of the proof of Proposition~\ref{prop:Valpha-step},
  rather than choosing the coordinates $\vf$ from all coordinates on
  $\C^n\times\cP^n_{\le k}$, we claim that it suffices to consider
  only coordinates on $\cP^n_{\le k}$ (i.e. coordinates of $\pi_2$),
  thereby obtaining a description of $\tilde Y(\Q,H;\pi_2)$ instead of
  $\tilde Y(\Q,H)$. This is permissible since the projection $\pi_2$
  has finite fibers when restricted to $\Sigma\cap\tilde\Omega$, and
  $\pi_2$ is therefore dominant on it as required. When we continue
  the induction $\Sigma$ is replaced by a collection of its
  irreducible subvarieties (and we may as well consider only those
  that meet $\tilde\Omega$), and the same argument applies. The rest
  of the inductive proof proceeds as in
  Lemma~\ref{lem:Valpha-induction}.
\end{proof}

The following lemma, in combination with~\eqref{eq:Xpoly-vs-Ytilde}
and Lemma~\ref{lem:Ytilde-blocks}, completes the proof of
Theorem~\ref{thm:main-k-complex}.

\begin{Lem}
  Let $\tilde V\subset \C^n\times\cP^n_{\le k}$ be an irreducible
  variety of degree $d$. Then there exists a collection of $d^{O(1)}$
  irreducible varieties $V_j\subset\C^n$ of degree $d^{O(1)}$ such that
  \begin{equation}
    \pi_1(\tilde Y(\tilde V_\alpha)) \subset \bigcup_j X(V_j).
  \end{equation}
\end{Lem}
\begin{proof}
  Write $V:=\pi_1(\tilde V)$. By \cite[Lemma~29]{me:rest-wilkie} there
  exist a hypersurface $\tilde\cH\subset\C^n\times\cP^n_{\le k}$
  (resp. $\cH\subset\C^n$) containing $\Sing\tilde V$ (resp.
  $\Sing V$) and not containing $\tilde V$ (resp. $V$). Proceeding by
  induction over dimension for the irreducible components of
  $\tilde V\cap\tilde\cH$ and $\tilde V\cap\pi_1^{-1}(\cH)$, we obtain
  a collection of varieties $V_j'$ such that
  \begin{equation}
    \pi_1(\tilde Y(\Sing\tilde V))   \cup
    \pi_1(\tilde Y(\tilde V\cap\pi_1^{-1}(\Sing\tilde V)) \subset \bigcup_j X(V'_j).
  \end{equation}
  It is easy to verify inductively that size and degrees of the
  collection $V_j'$ satisfy the required asymptotic estimates. To
  complete the construction, we let $\tilde p\in\tilde Y(\tilde V)$
  and suppose that $\tilde p\not\in\Sing\tilde V$ and
  $p:=\pi_1(\tilde p)\not\in\Sing V$. We claim that in this case
  $p\in X(V)$ (so taking the collection $V_j'$ in addition to $V$
  completes the proof).

  Fix a small ball $\tilde B\subset\C^n\times\cP^n_{\le k}$ around
  $\tilde p$ such that $\tilde V$ is smooth in $\tilde B$ and $V$ is
  smooth in $B:=\pi_1(\tilde B)$. By the Sard theorem applied to
  $\pi_1\rest{\tilde B\cap\tilde V}$, we may find a point
  $\tilde q\in\tilde B\cap\tilde V$ arbitrarily close to $\tilde p$
  which is a non-critical point of $\pi_1\rest{\tilde B\cap\tilde V}$.
  In particular $\pi_1\rest{\tilde B\cap\tilde V}$ is submersive at
  $\tilde q$, so there exists a neighborhood
  $U_{\tilde q}\subset\tilde B\cap\tilde V$ of $\tilde q$ such that
  $U_q:=\pi_1(U_{\tilde q})\subset B\cap V$ is a neighborhood of
  $q:=\pi_1(\tilde q)$ in $B\cap V$. Now since
  $\tilde p\in\tilde Y(\tilde V)$, we may assume (for an appropriate
  choice of $\tilde q$) that $U_{\tilde q}\subset\tilde Y$. Then by
  definition of $\tilde Y$ it follows that
  $U_q=\pi_1(U_{\tilde q})\subset X$.

  In conclusion, we see that $X$ contains the germ of $V$ at points
  $q$ arbitrarily close to $p$. Since the germ of $V$ at $p$ is
  irreducible (in fact smooth) and $X$ is analytic, it follows that
  $X$ contains the germ of $V$ at $p$, i.e. $p\in X(V)$ as claimed.
\end{proof}

\subsection{The main result in the real setting}

Finally we are ready to conclude the proof of Theorem~\ref{thm:main-k}
by reduction to the case of (complex) Noetherian varieties.

\begin{proof}[Proof of Theorem~\ref{thm:main-k}.]
  It clearly suffices to consider the case of a single basic
  semi-Noetherian set. Next, one can easily reduce to the case of
  Noetherian varieties by dropping all inequalities. Indeed suppose
  that $X=Y\cap U$ where $Y$ is a Noetherian variety and $U$ is
  defined by a collection of strict Noetherian inequalities. Then if
  $\{S_\alpha\}$ is a collection constructed for $Y$ as in the conclusion
  of Theorem~\ref{thm:main-k}, we have
  \begin{equation}
    X(k,H)\subset X\cap Y(k,H)\subset X\cap\big[\bigcup_{\alpha} Y(S_\alpha)\big]\subset\bigcup_{\alpha} X(S_\alpha)
  \end{equation}
  where the final inclusion follows since $X$ is locally open in $Y$,
  see e.g. \cite[Lemma~26]{me:rest-wilkie}. Henceforth we assume that
  $X$ is a real Noetherian variety of degree $\beta$.

  Recall from~\secref{sec:complex-noetherian} that the real Noetherian
  chain used in the definition of $X$ admits holomorphic continuation
  to a complex Noetherian chain in a domain $\Omega\supset\Omega_\R$
  with $\Omega\cap\R^n=\Omega_\R$, and that the Noetherian size of
  this chain is at most twice the Noetherian size of the real chain.
  We let $\tilde X\subset\Omega$ denote the complex Noetherian variety
  defined by the (holomorphic continuations of) the real Noetherian
  functions defining $X$, so that $X=\tilde X\cap\R^n$.

  Now let $\{\tilde V_\alpha\}$ denote the collection constructed for
  $\tilde X$ as in Theorem~\ref{thm:main-k-complex} and set
  $V_\alpha:=\R^n\cap\tilde V_\alpha$. Then $V_\alpha$ is cut out by
  the equations of $V_\alpha$ in addition to linear equations for the
  vanishing of the imaginary parts, and in particular has complexity
  bounded by $O(\beta)$. By \cite[Theorem~2]{gv:strata} one can
  decompose $V_\alpha$ into a union of $\beta_2:=\beta_1^{2^{O(m)}}$
  smooth (but not necessarily connected) semialgebraic sets of
  complexity $\beta_2$. Finally, by \cite[Theorem~16.13]{basu:book}
  each such semialgebraic set can be decomposed into its connected
  components, with the number of connected components bounded by
  $\beta_3=\beta_2^{O(m^4)}$ and their complexity bounded by
  $\beta_3$. We let $\{S_\eta\}$ denote the union of the collections of
  these components for every $V_\alpha$. One then easily verifies
  that $\{S_\eta\}$ satisfies the stated asymptotic estimates for
  the size and complexity, and finally we have
  \begin{equation}
    X(k,H) \subset \R^n\cap\tilde X(k,H)\subset\R^n\cap\big[\bigcup_{\alpha} \tilde X(\tilde V_\alpha)\big]
    \subset \bigcup_{\alpha} X(V_\alpha) \subset \bigcup_\eta X(S_\eta).
  \end{equation}
\end{proof}

\bibliographystyle{plain} \bibliography{nrefs}

\end{document}